\documentclass[11pt,a4paper]{article}
\usepackage[cp1251]{inputenc}
\usepackage{amsmath,amsthm}
\usepackage{amsfonts,amstext,amssymb,verbatim,epsfig}
\usepackage{dsfont}
\usepackage{psfrag}
\usepackage{graphicx}
\usepackage{enumerate}
\usepackage{enumitem}
\usepackage{color}

\usepackage{anysize}
\marginsize{1in}{1in}{1in}{0.9in}

\def\R{{\mathbb{R}}}
\def\N{{\mathbb{N}}}
\def\Z{{\mathbb{Z}}}

\renewcommand{\P}{\mathbb{P}}

\newcommand{\capa}{\mathrm{cap}} 
\newcommand{\local}{\mathcal{L}} 
\newcommand{\ZZ}{\mathbf{Z}}  

\newcommand{\cube}{\mathrm{Q}} 
\newcommand{\edges}{\square}   
\newcommand{\dcube}{{\partial_{\mathrm{int}}\mathrm{Q}}}
\newcommand{\genl}{\ell}  
\newcommand{\ballZ}{\mathrm{B}}  
\newcommand{\ballZZ}{\mathbf{B}} 

\newcommand{\sigmagood}{\sigma({\Sigma_{\scriptscriptstyle{ \mathcal{G},N}}})}
\newcommand{\extb}{{\partial_{\mathrm{ext}}}}  
\newcommand{\intb}{{\partial_{\mathrm{int}}}}  
\newcommand{\inn}{\mathrm{in}}
\newcommand{\outt}{\mathrm{out}}
\newcommand{\salgebra}{{\mathcal F_{\Omega}}} 
\newcommand{\setW}{\mathrm{W}} 
\newcommand{\intbb}{\partial_{\mathrm{\bf int}}} 
\newcommand{\Sigmagood}{{\Sigma_{\scriptscriptstyle{ \mathcal{G},N}}}}

\newcommand{\Sigmaall}{\Sigma}
\newcommand{\sigmagoodel}{{\sigma_{\scriptscriptstyle{ \mathcal{G},N}}}}
\newcommand{\sigmaallel}{\sigma}
\newcommand{\Sigmagoodout}{{\Sigma^\outt_{\scriptscriptstyle{ \mathcal{G},N}}}}
\newcommand{\sigmagoodelout}{{\sigma^\outt_{\scriptscriptstyle{ \mathcal{G},N}}}}

\newcommand{\loczd}{\N^{\Z^d}} 
\newcommand{\loczdsig}{\mathcal{F}_{\ell}} 

\definecolor{Red}{rgb}{1,0,0}

\newtheorem{theorem}{Theorem}[section]
\newtheorem{corollary}[theorem]{Corollary}
\newtheorem{lemma}[theorem]{Lemma}

\newtheoremstyle{likedef}
  {}%
  {}%
  {}%
  {\parindent}%
  {\bfseries}%
  {.}%
  {.5em}%
  {}%

\theoremstyle{likedef}
\newtheorem{definition}[theorem]{Definition}
\newtheorem{remark}[theorem]{Remark}
\newtheorem{claim}[theorem]{Claim}

\numberwithin{equation}{section}

\begin{document}

\title{Local percolative properties of the vacant set of \\  random interlacements with small intensity}

\author{
Alexander Drewitz
\thanks{Columbia University, Department of Mathematics, RM 614, MC 4419,
2990 Broadway, New York City, NY 10027, USA. email: drewitz@math.columbia.edu}
\and
Bal\'azs R\'ath
\thanks{
The University of British Columbia, Department of Mathematics,
Room 121, 1984 Mathematics Road, Vancouver, B.C.
Canada V6T 1Z2. email: rathb@math.ubc.ca}
\and
Art\"{e}m Sapozhnikov
\thanks{Max-Planck Institute for Mathematics in the Sciences, 
Inselstrasse 22, 04103 Leipzig, Germany. 
email: artem.sapozhnikov@mis.mpg.de}
}

\maketitle


\begin{abstract}
Random interlacements at level $u$  is a one parameter family of connected
random  subsets of $\Z^d$, $d \geq 3$ \cite{SznitmanAM}. 
Its complement, the vacant set at level $u$, exhibits a non-trivial percolation
phase transition in $u$ \cite{SznitmanAM,SidSzn_cpam}, 
and the infinite connected component, when it exists, is almost surely unique \cite{Teixeira_AAP}. 

In this paper we study local percolative properties of the vacant set of random interlacements at level $u$ 
for all dimensions $d\geq 3$ and small intensity parameter $u>0$.
We give a stretched exponential bound on the probability that a large (hyper)cube
contains two distinct macroscopic components of the vacant set at level $u$. 
In particular, this implies that finite connected components of the vacant set at level $u$ are unlikely to be large. 
These results are new for $d\in\{3,4\}$. 
The case of $d\geq 5$ was treated in \cite{Teixeira}
by a method that crucially relies on
a certain ``sausage decomposition'' of the trace of a high-dimensional bi-infinite random walk.
Our approach is independent from that of \cite{Teixeira}. 
It only exploits basic properties of random walks, such as Green function estimates and Markov property, 
and, as a result, applies also to the more challenging low-dimensional cases. 
One of the main ingredients in the proof is a certain conditional independence property of the random interlacements, 
which is interesting in its own right. 
\end{abstract}

%
%

\section{Introduction}

Random interlacements $\mathcal I^u$ at level $u>0$ on $\Z^d$, $d\geq 3$,  
is a one parameter family of random connected subsets of $\Z^d$, 
introduced by Sznitman \cite{SznitmanAM}, which
 arises as the local limit as $N \to \infty$ of the set of sites visited by a simple random walk on
the discrete torus $(\Z/N\Z)^d$, $d \geq 3$ when it runs up to time $\lfloor u  N^d \rfloor$, 
 see \cite{windisch_torus}. 
The law of $\mathcal I^u\subseteq \Z^d$ is uniquely characterized by the equations:
\begin{equation}\label{def:Iu}
\mathbb P[\mathcal I^u\cap K = \emptyset] = e^{-u\cdot \capa(K)}, \quad\text{for any finite $K\subseteq\Z^d$,}
\end{equation}
where $\capa(K)$ denotes the discrete capacity of $K$, defined in \eqref{def:capacity} below. 
It is proved among other results in \cite{SznitmanAM} that for any $u>0$, 
$\mathcal I^u$ is almost surely connected, and 
its law is invariant and ergodic with respect to the lattice shifts. 
In fact, in \cite{SznitmanAM}, a more constructive definition of $\mathcal I^u$ is given, 
which we recall in Section~\ref{sec:randominterlacements}. 
Informally, it states that $\mathcal I^u$ is the trace of a certain cloud of bi-infinite random walk trajectories in $\Z^d$, 
with $u$ measuring the density of this cloud. 

The vacant set $\mathcal V^u$ at level $u$ is the complement of $\mathcal I^u$ in $\Z^d$. 
We view $\mathcal V^u$ as a random graph 
by drawing an edge between any two vertices of the vacant set at $L_1$-distance $1$ 
from each other. 
The vacant set exhibits a non-trivial structural phase transition in $u$, i.e.,   
 there exists $u_*\in(0,\infty)$ such that 
\begin{itemize}\itemsep1pt
\item[(i)] 
for any $u>u_*$, almost surely, all connected components of $\mathcal V^u$  are finite, and 
\item[(ii)]
for any $u<u_*$, almost surely, $\mathcal V^u$ contains an infinite connected component. 
\end{itemize}
In particular, the finiteness of $u_*$ for $d \geq 3$ and
the positivity of $u_*$ for $d \geq 7$ were proved in \cite{SznitmanAM}, and
the latter result was extended to all dimensions $d \geq 3$ in \cite{SidSzn_cpam}.
It is also known that $\mathcal V^u$ contains at most one infinite connected component (see \cite{Teixeira_AAP});
in particular, for any $u<u_*$, the infinite connected component is almost surely unique. 

\medskip

In this paper, we are interested in the local structure of the vacant set in the regime of small $u$. 
More specifically, we show that with high probability, 
the unique infinite connected component of $\mathcal V^u$ is ``visible'' in 
large hypercubic subsets of $\Z^d$ (as the unique macroscopic connected component 
in the restriction of $\mathcal V^u$ to large hypercubes of $\Z^d$). 
Our main result is the following theorem.
\begin{theorem}[Local uniqueness for $\mathcal V^u$] \label{thm:sts:interlacement}
For any $d\geq 3$, there exist $u_1>0$, $c = c(d)>0$ and $C = C(d)<\infty$ such that 
for all $0 \leq u \leq u_1$ and $n\geq 1$, we have
\begin{equation}\label{eq:sts:interlacement:1}
\mathbb P\left[
\begin{array}{c}
\text{the infinite connected component of $\mathcal V^u$}\\
\text{intersects $\ballZ(0,n)$}
\end{array}
\right] \geq 1 - Ce^{-n^c} 
\end{equation}
and
\begin{equation}\label{eq:sts:inerlacement:2}
\mathbb P\left[
\begin{array}{c}
\text{any two connected subsets of $\mathcal V^u\cap\ballZ(0,n)$ with}\\
\text{diameter $\geq n/10$ are connected in $\mathcal V^u\cap\ballZ(0,2n)$}
\end{array}
\right]
\geq 1 - Ce^{-n^c} .\
\end{equation}
\end{theorem}
Statement \eqref{eq:sts:interlacement:1} has already been known 
(it easily follows from \cite[Theorem 5.1]{Sznitman:Decoupling}), 
but we include it here for completeness. 
For $d\geq 5$, statement \eqref{eq:sts:inerlacement:2} follows from 
the stronger statement of \cite[Theorem~3.2]{Teixeira}. 
Our contribution to the result of Theorem~\ref{thm:sts:interlacement} is twofold. 
Firstly, the result \eqref{eq:sts:inerlacement:2} is new for $d\in\{3,4\}$. 
Secondly, our proof of \eqref{eq:sts:inerlacement:2} is conceptually different from that of \cite{Teixeira}, and 
applies to all dimensions $d\geq 3$. 
Let us briefly explain the strategy in the proof of \cite{Teixeira} and why it cannot be used in low dimensions.
The proof in \cite{Teixeira} crucially relies on the 
fact that if $d \geq 5$, the trace of a bi-infinite random walk contains
many  bilateral cut-points (see \cite[(6.1),(6.26)]{Teixeira}). 
This gives a decomposition of the random walk trace into a chain of relatively small well-separated ``sausages''. 
Heuristically, a chain of sausages cannot separate two macroscopic connected subsets of a box.
Random interlacements at level $u$ is the trace of a certain Poisson cloud of doubly infinite random walk trajectories in $\Z^d$, 
and, therefore, can be viewed as the countable union of doubly infinite chains of ``sausages'' in $\Z^d$. 
Thus, in order to show that random interlacements at level $u$ cannot separate two macroscopic connected subsets of a large box, 
one needs to show that locally it generally looks like the trace of only bounded number of random walks. 
This is achieved in \cite{Teixeira} with a renormalization argument. 
The sausage decomposition property fails for $d\leq 4$ (see, e.g., \cite[Theorem~2.6]{Lawler80}).
In fact, in dimension $d=3$, even the trace of a single random walk is a ``two-dimensional'' object, and, therefore,  
could in principle form a large separating surface in a box. 
This is not the case, as we discuss in Section~\ref{sec:bonus}.
Our proof of \eqref{eq:sts:inerlacement:2} 
only exploits basic properties of random walks (Green function estimates, Markov property) 
and works for all dimensions $d\geq 3$.

\medskip

The results of Theorem~\ref{thm:sts:interlacement} are in the spirit of the local uniqueness 
property of supercritical Bernoulli percolation (see, e.g., \cite[(7.89)]{Grimmett}). 
In fact, the analogues of \eqref{eq:sts:interlacement:1} and \eqref{eq:sts:inerlacement:2} for Bernoulli percolation
hold through the whole supercritical phase. 
We believe that the bounds \eqref{eq:sts:interlacement:1} and \eqref{eq:sts:inerlacement:2} also hold 
for all $u<u_*$, but with constants $c = c(d,u)>0$ and $C = C(d,u)<\infty$ depending on $u$. 
Our current understanding of the model is not good enough to be able to rigorously justify this belief. 

The main technical challenges in the proof of Theorem~\ref{thm:sts:interlacement} 
come from the long-range dependence of the random interlacements (see, e.g., \cite[Remark 1.6 (4)]{SznitmanAM}), 
the lack of the BK inequality (see, e.g., \cite[(2.12)]{Grimmett} and \cite[Remark 1.5 (3)]{Sznitman:Decoupling}) 
and the absence of finite energy property (see, e.g., \cite[Remark~2.2 (3)]{SznitmanAM}).

\medskip

As an immediate corollary of Theorem~\ref{thm:sts:interlacement} we obtain that 
finite connected components of the vacant set at level $u$ are unlikely to be large when $u$ is small enough. 
\begin{corollary}\label{cor:finitecluster}
For any $d\geq 3$, there exist $c = c(d)>0$ and $C = C(d)<\infty$ such that 
for all $u\leq u_1$ (defined in Theorem~\ref{thm:sts:interlacement}), we have 
\begin{equation}\label{eq:Cdiameter}
\mathbb P\left[n\leq \mathrm{diam}(\mathcal C^u(0)) < \infty \right] \leq C e^{-n^c} 
\end{equation}
and 
\begin{equation}\label{eq:Csize}
\mathbb P\left[n\leq |\mathcal C^u(0)| < \infty\right] \leq C e^{-n^c} ,\
\end{equation}
where $\mathrm{diam}(\mathcal C^u(0))$ and $|\mathcal C^u(0)|$ denote the diameter and the cardinality
of the connected component of the origin in $\mathcal V^u$, respectively.
\end{corollary}
Again, when $d\geq 5$, the result of Corollary~\ref{cor:finitecluster} follows from \cite[Theorems 3.5 and 3.6]{Teixeira}. 
The analogue of Corollary~\ref{cor:finitecluster} for supercritical Bernoulli percolation is well known, 
and as Theorem~\ref{thm:sts:interlacement}, it is a property of the whole supercritical phase of Bernoulli percolation 
(see, e.g., \cite{Chayes_Chayes_Newman_87}, \cite{KestenZhang} and \cite[Chapter~8]{Grimmett}). 
Moreover, the analogue of \eqref{eq:Cdiameter} for Bernoulli percolation holds with exponential decay rate 
(see, \cite[(8.20)]{Grimmett}), 
and the analogue of \eqref{eq:Csize} holds with stretched exponential decay with the explicit exponent $c = (d-1)/d$ 
(see, e.g., \cite[(8.66)]{Grimmett}). 

\medskip

Let us now mention some applications of Theorem~\ref{thm:sts:interlacement}. 
In \cite{RS:BP}, Theorem~\ref{thm:sts:interlacement} is used to study 
the stability of the phase transition of the vacant set under a small quenched noise. 
The setup is the following. 
For a positive $\varepsilon$, we allow each vertex of the random interlacement (referred to as occupied) to become vacant, 
and each vertex of the vacant set to become occupied with probability $\varepsilon$, 
independently of the randomness of the interlacement, and independently for different vertices.
In \cite[Theorem~5]{RS:BP} it is proved that for any $u$ which satisfies \eqref{eq:sts:interlacement:1}
 and \eqref{eq:sts:inerlacement:2}, 
the perturbed vacant set at level $u$ still has an infinite connected component if the noise is small enough. 
In particular, this statement together with Theorem~\ref{thm:sts:interlacement} imply that 
the perturbed vacant set at small level $u$ still has an infinite connected component. 
The use of Theorem~\ref{thm:sts:interlacement}
significantly simplifies the original proof of \cite[Theorems~3 and 5]{RS:BP} given 
in the first version of \cite{RS:BP}.  

In \cite[Theorem~2.3]{DRS}, we use Theorem~\ref{thm:sts:interlacement} as an ingredient to prove that 
the graph distance in the unique infinite connected component of the vacant set at small level $u$ is 
comparable to the graph distance on $\Z^d$, 
and establish a shape theorem for balls with respect to graph distance on the infinite connected component. 

We  believe that the methods of this paper 
can be applied in order to further explore the fragmentation of the torus $(\Z/N\Z)^d$ by the trace of a simple
random walk, in a similar fashion to \cite{teixeira_windisch}, where a strong coupling between the random
walk trace on the torus and random interlacements is used to transfer
 results of \cite{Teixeira} to the torus. We further discuss
this possibility as well as the analogue of Theorem \ref{thm:sts:interlacement}
for the set of sites avoided by a simple random walk on $\Z^d$ 
in Section~\ref{sec:bonus}.

\bigskip

We will now briefly sketch the main ideas of the proof of Theorem~\ref{thm:sts:interlacement}. 
A more detailed description of the main steps of the proof will be given 
at the beginning of Sections \ref{sec:coarsegraining}, \ref{sec:condindep}, and \ref{sec:proofthm}. 
Before reading those descriptions, we advise the reader to become familiar with 
basic definitions and results concerning random interlacements in 
Sections \ref{sec:randominterlacements} and \ref{subsection_local_times}. 

The proof uses coarse graining (see Section~\ref{sec:coarsegraining}) and 
a conditional independence property for random interlacements (see Section~\ref{sec:condindep}). 
The need for coarse graining comes from the fact that the complement of the 
infinite connected component of the vacant set is almost surely connected, no matter how small the parameter $u$ is. 
(This is immediate from the fact that $\mathcal I^u$ is almost surely connected for any given $u$,
 see \cite[(2.21)]{SznitmanAM}.)
The reader familiar with Bernoulli percolation may notice that this would not be the case if
 the vertices were made vacant 
independently from each other. 
In this case, the usual Peierls argument would easily give the analogue of Theorem~\ref{thm:sts:interlacement} 
for Bernoulli percolation, when the vacant set has  density close to one.

To overcome the problem arising from the connectedness of $\mathcal I^u$, 
we partition $\Z^d$ into $L_\infty$-boxes $(\ballZ(x',R)~:~ x'\in(2R+1)\cdot\Z^d)$, with some $R\geq 0$. 
We use a variant of Sznitman's decoupling inequalities \cite{Sznitman:Decoupling} to show 
that when $R$ is large enough, there is a unique infinite connected subset of good boxes 
which are ``sufficiently vacant''. 
Moreover, the remaining (bad) boxes form only finite connected subsets of $\Z^d$, 
with stretched exponential decay of the probability that a connected component of bad boxes is large. 
Our definition of good boxes also assures that 
the infinite connected component of good boxes contains an infinite connected subset of $\mathcal V^u$, 
which intersects every good box of the above set. 
For concreteness, in this proof sketch, we call this infinite connected subset of $\mathcal V^u$ the ``fat'' set. 
As a result, we obtain that with high probability, 
any nearest-neighbor path of $\Z^d$ with large diameter 
often intersects the infinite connected component of good boxes, and therefore
gets within distance $R$ from the fat set. 

However, the possibility of having a long nearest-neighbor path in $\mathcal V^u$ which avoids the fat set 
(but unavoidably, with high probability, gets $R$-close to it sufficiently often) still remains. 
We use a conditional independence property of random interlacements (see Section~\ref{sec:condindep}) 
to show  that, roughly speaking, conditionally on  the fact that a vacant path connects to 
a good box of the infinite connected set of good boxes and also conditioning on the configuration outside this box, 
there is still a uniformly positive chance that this vacant path is connected inside the specified 
good box to the fat set.
The difficulty in the proof of this claim comes from the fact that 
random interlacements do not posess the so-called finite energy property 
(see, e.g., \cite[Remark~2.2 (3)]{SznitmanAM}). 
In words, the fact that $\mathcal I^u$ is a connected set implies that 
depending on the realization of $\mathcal I^u$ outside a box, 
not every configuration can be realized by $\mathcal I^u$ inside this box. 
(This is a big constraint, and, for example, causes some difficulties in the proof of the uniqueness of an
infinite connected component of $\mathcal V^u$, see \cite{Teixeira_AAP}.)
Our definition of good boxes  
 is chosen specifically to 
overcome this problem. 
Coming back to the proof sketch, 
since each long path must visit many good boxes in the infinite connected component, 
we conclude  that 
with high probability each long path in $\mathcal V^u$ must be connected to the fat set. 
This gives us \eqref{eq:sts:inerlacement:2}.

\bigskip

The paper is organized as follows. 
In Section~\ref{sec:notation}, we define the notation used in the paper, 
state some basic results about the simple random walk on $\Z^d$, 
define random interlacements and recall some of its properties, 
the most important of which is Lemma~\ref{l:probability:cascading}. 
It is based on \cite[Corollary~3.5]{Sznitman:Decoupling}, but formulated more generally (using so-called 
interlacement local times defined in Section \ref{subsection_local_times}). 
Therefore, we give its proof sketch in the Appendix. 

In Section~\ref{sec:coarsegraining}, we define coarse graining, and prove the existence of a ``fat'' infinite 
connected subset of $\mathcal V^u$, when $u$ is small enough (see Corollary~\ref{cor:bad:*paths}). 

In Section~\ref{sec:condindep}, we prove a conditional independence property of random interlacements 
(see Lemma~\ref{l:condindep}). 

In Section~\ref{sec:proofthm}, we prove Theorem~\ref{thm:sts:interlacement} 
using the results of Sections \ref{sec:coarsegraining} and \ref{sec:condindep}. 

Finally, in Section~\ref{sec:bonus}, we briefly mention applications of the ideas developed in this paper to 
the vacant set of a simple random walk on $\Z^d$ and $(\Z/N\Z)^d$.

\section{Notation, model, preliminaries}\label{sec:notation}

\subsection{Basic notation}\label{subsection_basic_notation}

We denote by $\N=\{0,1,\dots\}$ the set of natural numbers, by $\Z$ the set of integers. 
We denote by $\mathbb R$ the set of real numbers and by $\mathbb R_+$ the set of non-negative reals. 
For $a\in\R$, we write $|a|$ for the absolute value of $a$, and $\lfloor a\rfloor$ for
 the integer part of $a$.

For any $d \geq 1$,  we denote by $x=(x_1,\dots,x_d)$ a generic element 
 of $\Z^d$, also referred to as \emph{vertex} of $\Z^d$.
We denote by $|x|=\max_{ 1\leq i\leq d }|x_i|$ the sup-norm of $x \in \Z^d$ 
and by $|x|_1=\sum_{i=1}^d |x_i|$ the $L_1$-norm of $x$.
For $K \subset \Z^d$, we denote by $|K|$ the cardinality of $K$.
We write $K \subset \subset \Z^d$ when $K\subset \Z^d$ and $|K|<\infty$.  

We say that $x,x' \in \Z^d$ are nearest neighbors (respectively, $*$-neighbors) if $|x-x'|_1=1$
 (respectively, $|x-x'|=1$). We also denote $|x-x'|_1=1$ by $x \sim x'$.
We say that $\pi=(z_1,\dots,z_n)$ is a nearest neighbor path 
(respectively, $*$-path) if $z_i$ and $z_{i+1}$ are nearest neighbors (respectively, $*$-neighbors)
for all $1 \leq i \leq n-1$, and we use the notation $|\pi| = n$ 
(not to be confused with the cardinality of the set $\{z_1,\dots,z_n\}$). 
 We say that $V \subseteq \Z^d$ is connected (respectively, $*$-connected)
 if any pair $x_1,x_2 \in V$ can be connected by a nearest neighbor path (respectively, $*$-path)
 with vertices in $V$.

For $x \in \Z^d$ and $R \in \N$ 
we denote by $\ballZ(x,R)=\{ y \in \Z^d \, : \, |x-y|\leq R \}$ the closed ball of radius
$R$ around $x$ with respect to the sup-norm. For any set $V \subseteq \Z^d$, we denote by
$V^c = \Z^d \setminus V$.

The interior boundary of $K\subseteq\Z^d$, $\intb K$ is the set of vertices of $K$ that have some neighbor in $K^c$.

The exterior boundary of $K\subseteq\Z^d$, $\extb K$ is the set of vertices of $K^c$ that have some neighbor in $K$.

Given a probability space $(\Omega, \mathcal{F}, \P)$ and $A \in \mathcal{F}$,  
we denote by $\mathds{1}_A$ the indicator of the event $A$. If $X$ is an integrable random variable on
$(\Omega, \mathcal{F}, \P)$, we denote  $\mathbb{E}[X ; A]=\mathbb{E}[X\cdot \mathds{1}_A ]$. 

For $-\infty \leq a < b \leq +\infty$, we denote by $\mathcal B([a,b])$ the Borel $\sigma$-algebra on $[a,b]$. 

Our agreement about the constants used in the paper is the following. 
We denote small positive constants by $c$ and large finite constants by $C$. 
When needed, we emphasize the dependence of a constant on parameters. 
If the constant only depends on $d$, then we sometimes do not mention it at all. 
The value of a constant may change within the same formula.

\subsection{Simple random walk and potential theory}

The space $W_+$ stands for the set of infinite nearest-neighbor trajectories, 
defined for non-negative times and tending to infinity:
\begin{equation}\label{def_eq_W_plus}
W_+ = \big\{ w : \mathbb{N} \rightarrow \Z^d, \; \;  w (n) \sim w (n+1), \;
n \in \N, \; \;  \lim_{n \to \infty} |w(n)|=\infty   \big\}.\\
\end{equation}
We endow $W_+$ with the $\sigma$-algebra $\mathcal{W}_+$ generated by the canonical coordinate maps $X_n$, $n \in \N$.
For each $k \in \N$,  we define the shift map $\theta_k: W_+ \rightarrow W_+$ by
$\theta_k(w)(\cdot) = w(\cdot + k)$.
For $x\in\Z^d$, let $P_x$ denote the law of simple random walk  on $\Z^d$ with starting point $x$. 
Simple random walk on $\Z^d$, $d \geq 3$, is transient and the set $W_+$ has full measure under any $P_x$.
From now on we will view $P_x$ as a measure on $(W_+, \mathcal{W}_+)$, and 
we write $(X(t)~:~t\in\N)$ for a random element of $W_+$ with distribution $P_x$. 

For $U \subseteq \Z^d$ and $w \in W_+$, we define
\begin{align}
H_U(w)  &=  \inf\{n \ge 0 \,:\, X_n(w) \in U\}, \quad \text{the entrance time in $U$,} \label{def:HU} \\
\widetilde{H}_U(w) &= \inf\{ n \ge 1 \, :\, X_n(w) \in U\}, \quad \text{the hitting time of $U$,}\label{def:tildeHU} \\
T_U(w)  &=  \inf\{n \ge 0\, :\, X_n(w) \notin U\}, \quad \text{the exit time from $U$.} \label{def:TU}
\end{align}

For $d\geq 3$, the Green function $g~:~\Z^d\times\Z^d\to[0,\infty)$ of
the simple random walk $X$ is defined as
\begin{equation*}
g(x,y) = \sum_{t=0}^\infty P_x[X(t) = y],~~x,y\in\Z^d .\
\end{equation*}
Translation invariance yields $g(x,y)=g(0,y-x)$. It follows from \cite[Theorem~1.5.4]{LawlerRW} that 
for any $d\geq 3$, there exist $c_g=c_g(d)>0$ and $C_g = C_g(d)<\infty$ such that 
\begin{equation}\label{Green_decay}
 c_g \cdot (|x-y|+1)^{2-d} \leq  g(x,y) \leq C_g \cdot (|x-y|+1)^{2-d}, \quad \text{for $x,y\in\Z^d$.}
\end{equation}

\bigskip

\noindent
The equilibrium measure of $K \subset \subset \Z^d$ is defined by 
\begin{equation*}
 e_K(x) = 
\left\{
\begin{array}{l}
P_x\big[\widetilde H_K = \infty\big],\quad x\in K ,\\
0,\quad x\notin K .\
\end{array}
\right.
\end{equation*}
The capacity of $K$ is the total mass of the equilibrium measure of $K$:
\begin{equation}\label{def:capacity}
 \mathrm{cap}(K) = \sum_x e_K(x) .\
\end{equation}
Since $\Z^d$ is transient ($d\geq 3$), 
for any $\emptyset\neq K\subset \subset \Z^d$, the capacity of $K$ is positive. 
Therefore, we can define for such $K$ the normalized equilibrium measure by 
\begin{equation}\label{def:normalizedequilibriummeasure}
\widetilde e_K(x) = e_K(x)/\mathrm{cap}(K) .\
\end{equation}

The following relations for $P_x[H_K<\infty]$ will be useful: 
for any $K\subset\subset \Z^d$ and $x\in\Z^d$, 
\begin{itemize}\itemsep1pt
\item[(i)](see, e.g. \cite[(1.8)]{SznitmanAM})
\begin{equation}\label{eq:hittingprobability:equality}
P_x[H_K<\infty] = \sum_{y\in K} g(x,y) e_K(y) ,\
\end{equation}
\item[(ii)](see \cite[(1.9)]{SznitmanAM})
\begin{equation}\label{eq:hittingprobability:bounds}
\sum_{y\in K} g(x,y) / \sup_{z\in K}\sum_{y\in K}g(z,y)
\leq 
P_x[H_K<\infty]
\leq 
\sum_{y\in K} g(x,y) / \inf_{z\in K}\sum_{y\in K}g(z,y) .\
\end{equation}
\end{itemize}

\subsection{Definition of random interlacements}\label{sec:randominterlacements}

Now we recall the definition of the interlacement point process from \cite[Section 1]{SznitmanAM}.
We consider the space of doubly infinite nearest-neighbor trajectories $W$:
\begin{equation}\label{def:W}
W = \big\{ w : \Z \rightarrow \Z^d, \; \;  w (n) \sim w (n+1), \;
n \in \Z, \; \;  \lim_{n \to \pm \infty} |w(n)|=\infty   \big\}.\\
\end{equation}
We endow $W$ with the $\sigma$-algebra
 $\mathcal{W}$ generated by the coordinate maps $X_n$, $n \in \Z$.

Consider the space $W^*$ of trajectories in $W$ modulo time shift
\begin{equation*}
W^* = W / \sim, \text{ where } w \sim w' \iff w(\cdot) = w'(\cdot + k) \text{ for some } k \in \Z. 
\end{equation*}
and denote by $\pi^*$ the canonical projection from $W$ to $W^*$ which assigns to each $w \in W$ the 
$\sim$-equivalence class $\pi^*(w)$ of $w$.
The map $\pi^*$ induces a $\sigma$-algebra on $W^*$ given 
by $\mathcal{W}^* = \{A \subset W^*~:~(\pi^*)^{-1}(A) \in \mathcal{W}\}$.

For $K \subset \subset \Z^d$, we denote by $W_K$ the set of trajectories in
$W$ that enter the set $K$, and denote by $W_K^*$ the image of $W_K$ under $\pi^*$.
Note that $W_K\in\mathcal W$ and $W_K^*\in\mathcal{W}^*$.

For any $w^* \in W^*$ and $u \in \R_+$  we call the pair $(w^*,u)$ a labeled trajectory.
The space of point measures on which one canonically defines  random interlacements is given by
\begin{equation}
\label{eq:omega}
\Omega = \left\{\omega = \sum_{i\geqslant 1} \delta_{(w^*_i,u_i)}~:~  w^*_i \in W^*, u_i \in \mathbb{R}_+ 
\mbox{ and } \forall \, K \subset \subset \Z^d, u \geq 0\; : \; \omega(W^*_K \times [0,u]) < \infty \right\}.
\end{equation}
The space $\Omega$ is endowed with the $\sigma$-algebra $\mathcal{F}_{\Omega}$ generated by the evaluation
 maps of form $\omega \mapsto \omega(D)$
for $D \in \mathcal{W}^* \otimes \mathcal{B}(\R_+)$.
We recall the definition of the measure $Q_K$ on $(W, \mathcal{W})$ from \cite[(1.24)]{SznitmanAM}:
for any $A,B \in \mathcal{W}_+$ and $x \in \Z^d$ let
\begin{equation}\label{def_Q_K}
Q_K[(X_{-n})_{n\geq 0} \in A,~ X_0=x,~ (X_n)_{n \geq0} \in B]=
P_x[A \,|\, \widetilde{H}_K=\infty] 
\cdot e_K(x) \cdot P_x[B].
\end{equation}
According to \cite[Theorem 1.1]{SznitmanAM}, 
there exists a unique $\sigma$-finite measure $\nu$ on $(W^*,\mathcal{W}^*)$
which satisfies the identity
\begin{equation}\label{def_eq_nu}
\nu(E) = Q_K[ (\pi^*)^{-1}(E)],\quad
\text{for all $K \subset \subset \Z^d$ and $E \in \mathcal{W}^*$ with $E \subseteq W^*_K$.}
\end{equation}

\medskip

The \emph{interlacement point process}  is
 the Poisson point process on $W^* \times \mathbb{R}_+$
with intensity measure $\nu(\mathrm{d}w^*) \mathrm{d}u$, defined on the  probability space $(\Omega, \salgebra, \mathbb{P})$. 
Given $\omega = \sum_{i \geq 1} \delta_{(w^*_i,u_i)}  \in \Omega$ and $u\geq 0$, 
the {\it random interlacement at level $u$} is the random subset of $\Z^d$
 defined by
\begin{equation}\label{def_eq_interlacement_at_level_u}
\begin{array}{l}
\mathcal{I}^u(\omega) = 
\bigcup\limits_{ i \geq 1, \, u_i < u} {\rm range}(w^*_i)\,,
\end{array}
\end{equation}

\noindent
where ${\rm range}(w^*) = \{ w(n)\, :\, n \in \Z \}$
 for any  $w \in \pi^{-1}(w^*)$. The {\it vacant set at level $u$} is defined as
\begin{equation*}
\mathcal{V}^u(\omega) = \Z^d \setminus \mathcal{I}^u(\omega), \quad \;\mbox{for} \;\omega \in \Omega, \,u \ge 0\,.
\end{equation*}
For the sake of consistency, we mention that the law of $\mathcal I^u$ is uniquely characterized by \eqref{def:Iu}, 
see \cite[Proposition~1.5 and Remark~2.2 (2)]{SznitmanAM}. 

\subsection{Discrete interlacement local times}\label{subsection_local_times}

In this section we define the interlacement local time field $\local^u(\omega)$ at level $u$, 
which counts the accumulated number of 
visits of the interlacement trajectories with label smaller than $u$
 to each vertex $x \in \Z^d$, see \eqref{def_eq_local_time_at_level_u}.
We introduce this notion so that we can control the number of excursions of 
the interlacement trajectories inside a box in Section~\ref{sec:proofthm}.

\medskip

We denote by $\genl$ a generic element of the product space $\loczd$. 
For any $x \in \Z^d$, denote by $\Psi_x: \loczd \to \N $ the 
canonical coordinate function defined by $\Psi_x(\genl) =\genl(x)$. 
We consider the measurable space $(\loczd, \loczdsig)$ where  $\loczdsig$ is the 
$\sigma$-algebra generated by the functions $\Psi_x$, $x \in \Z^d$. 
For $\genl, \genl' \in \loczd$, we say that $ \genl \leq \genl'$ if 
$\genl(x)\leq \genl'(x)$ for all $x\in\Z^d$. 
We say that an event $A \in \loczdsig$ is \emph{increasing} if for any 
$\genl, \genl' \in \loczd$
the conditions $\genl \in A$ and $\genl \leq \genl'$ imply $\genl' \in A$.

\noindent
Given $\omega= \sum_{i \geq 1} \delta_{(w^*_i,u_i)} \in \Omega$ and $u\geq 0$, we define
the discrete interlacement local time profile at level $u$, 
$\local^u(\omega)= \left( \local^u_x(\omega)~:~x \in \Z^d\right)$ as
\begin{equation}\label{def_eq_local_time_at_level_u}
\local^u_x(\omega)=  \sum_{ i \geq 1, \, u_i < u} \; \sum_{n \in \Z} \mathds{1}_{\{ w_i(n)=x \}}, \quad x\in\Z^d ,\
\end{equation}
where $w_i$ is any particular element of $\pi^{-1}(w^*_i)$.
Note that the function $\local^u~:~ (\Omega,\salgebra) \to (\loczd, \loczdsig)$ is measurable and that
$x \in \mathcal{I}^u(\omega)$ if and only if $\local^u_x(\omega) \geq 1$.

Given a measurable function $\local~:~ (\Omega,\salgebra) \to (\loczd, \loczdsig)$ 
and an event $A \in \loczdsig$, we define 
\begin{equation}\label{def:Au}
\text{$A(\local) = \{\omega \in \Omega ~:~ \local(\omega) \in A\}~~$ and $~~A^u = A(\local^u)$ for $u\geq 0$.}
\end{equation}
It follows from \eqref{def_eq_local_time_at_level_u} that for any $0\leq u \leq u'$,  
$\P[ \local^u \leq \local^{u'} ]=1$. 
Therefore, for any increasing event $A \in \loczdsig$ and $u \leq u'$, we have
\begin{equation}\label{eq:stochdomination}
\P[A^u] \leq \P[A^{u'}].
\end{equation}

\medskip

Finally, we record that for $x\in\Z^d$ and $u\geq 0$, 
\begin{equation}\label{eq:meanlocal}
\mathbb E[\local^u_x] = u .\
\end{equation}
Indeed, 
by \eqref{def_Q_K} and \eqref{def_eq_nu}, 
$\mathbb E[\local^u_x] = \mathbb E [ \omega(W^*_{\{x\}}\times [0,u) ) ]\cdot g(x,x) = \capa(\{x\}) \cdot u\cdot g(0,0) = u$.

\subsection{Cascading events}

In this section we adapt some results of \cite{Sznitman:Decoupling} 
to our setting which involves increasing events of $\loczd$. 
The result of Lemma~\ref{l:probability:cascading} below is new, but very similar to \cite[Corollary~3.5]{Sznitman:Decoupling}, 
which is stated for increasing events in $\{0,1\}^{G\times\Z}$, where $G$ is an infinite, connected, bounded degree weighted graph, 
satisfying certain regularity conditions (for example, $G = \Z^{d-1}$, with $d\geq 3$). 
We will use Lemma~\ref{l:probability:cascading} in the proof of Lemma~\ref{l:bad:*paths}.

\bigskip

We begin with the definition of uniformly cascading events. We adapt
 \cite[Definition~3.1]{Sznitman:Decoupling} to our setting which involves local times. 
\begin{definition}\label{def:cascading}
Let $\lambda>0$. 
We say that a family $\mathcal{G} = (G_{x,L,R})_{x\in \Z^d, L\geq 1,R\geq 0 }$ of events on $(\loczd, \loczdsig)$ 
{\it cascades uniformly} (in $R$) with complexity at most $\lambda>0$ if
there exists $C(\lambda)<\infty$ such that 
\begin{equation*}
\text{$G_{x,L,R}$ is $\sigma(\Psi_{y}, y\in \ballZ(x,10L))$-measurable 
for each $x\in \Z^d$, $R\geq 0$, and $L\geq 1$,} 
\end{equation*}
 and for each $l$ multiple of $100$, $x\in \Z^d$, $R\geq 0$, $L\geq 1$, there exists 
$\Lambda\subseteq \Z^d$ such that 
\begin{equation}\label{eq:cascading:2}
\Lambda\subseteq \ballZ(x,9lL) ,\
\end{equation}
\begin{equation}\label{eq:cascading:3}
|\Lambda| \leq C(\lambda)\cdot l^\lambda ,\
\end{equation}
\begin{equation}\label{eq:cascading:4}
G_{x,lL,R} \subseteq
\bigcup_{x_1,x_2\in\Lambda~:~|x_1-x_2|\geq \frac{l}{100}L}
G_{x_1,L,R}\cap G_{x_2,L,R} .\
\end{equation}
\end{definition}

\begin{lemma}\label{l:probability:cascading}
Let $\mathcal G= (G_{x,L,R})_{x\in \Z^d, L\geq 1, R\geq 0 }$ be a family 
of increasing events on $(\loczd, \loczdsig)$ cascading uniformly (in $R$) with complexity at most $\lambda>0$. 
\begin{equation}\label{def:scalesLn}
\text{Let $L_0 \geq 1$, $l_0$ large enough multiple of $100$, and $L_n = l_0^n L_0$.}
\end{equation}
Let $u_{L_0} = {L_0}^{2-d}$, and recall the notation of \eqref{def:Au}.
If 
\begin{equation}\label{eq:probability:cascading:1}
\inf_{R \geq 0,\,  L_0 \geq 1}~
\sup_{x\in \Z^d}~
\mathbb P\left[G_{x,L_0,R}^{u_{L_0}}\right] = 0 ,\
\end{equation}
then there exist $l_0>1$, $R\geq 0$, $L_0 \geq 1$ and $u>0$ such that 
\begin{equation}\label{eq:probability:cascading:2}
\sup_{x\in\Z^d} \mathbb P\left[G_{x,L_n,R}^u \right]
\leq 2^{-2^n} ,\quad \text{for all $n\geq 0$} .\
\end{equation}
\end{lemma}
The proof of Lemma~\ref{l:probability:cascading} is essentially the same as the proof of \cite[Corollary~3.5]{Sznitman:Decoupling}. 
For completeness, we include its sketch in the Appendix.

\bigskip

\section{Coarse graining of $\Z^d$}\label{sec:coarsegraining}

In this section we show that when $u$ is small enough, 
the infinite connected component of $\mathcal V^u$ contains a ubiquitous infinite connected subset, 
which has a well-prescribed structure and useful properties. 
We do so by partitioning $\Z^d$ into large boxes. 
We then define a notion of good boxes in Definition~\ref{def:good}. 
These boxes are defined to be ``sufficiently vacant''. 
In Lemma~\ref{l:bad:*paths}, we show that large $*$-connected components
of bad boxes are unlikely, where 
we use Lemma~\ref{l:probability:cascading} to deal with the long-range correlations present in the model. 
We then combine it with the result of \cite[Lemma~2.23]{Kesten} 
on the connectedness of the exterior $*$-boundary of a $*$-connected finite subset of $\Z^d$ 
to obtain 
in Corollary~\ref{cor:bad:*paths} that 
there is a unique infinite connected subset of good boxes (denoted by $\mathcal G^\infty$ in Corollary~\ref{cor:bad:*paths} (2)), 
and all the remaining bad components are very small. 
It then follows from the definition of good boxes that the infinite connected component of 
good boxes contains the desired infinite connected subset of $\mathcal V^u$ 
(see Corollary~\ref{cor:bad:*paths} (3)). 
An important consequence of Corollary~\ref{cor:bad:*paths}, which we will use in the proof of Theorem~\ref{thm:sts:interlacement} 
(see \eqref{eq:HN:proba} and \eqref{eq:Skproperty}), 
is that with high probability, any long nearest-neighbor path in $\Z^d$ will get within distance $R$ from the above defined 
infinite connected subset of $\mathcal V^u$ many times.

\subsection{Setup and auxiliary results}\label{sec:setup}

We consider the hypercubic lattice $\Z^d$ with $d\geq 3$. For an integer $R\geq 0$, let 
\begin{equation}\label{eq:ZZ}
\ZZ = (2R+1)\cdot \Z^d .\
\end{equation}
We say that $x',y'\in\ZZ$ are (1) nearest-neighbors in $\ZZ$, if $|x'-y'|_1 = 2R + 1$, and
(2) $*$-neighbors in $\ZZ$, if $|x'-y'| = 2R+1$. 
We denote by 
$\ballZZ(x',N) = \ballZ(x',(2R+1)N)\cap\ZZ$ the closed ball of radius $N$ in $\ZZ$.
The interior boundary  of $\mathbf K \subseteq \ZZ$, denoted by  $\intbb \mathbf K$, is
the set of vertices of $\mathbf K$ that have some nearest neighbor in $\ZZ \setminus \mathbf K$.
Note that for $R\neq 0$, the set $\intbb \mathbf K$ is different from $\intb \mathbf K$, 
defined in Section \ref{subsection_basic_notation}.

With each vertex $x'\in \ZZ$, we associate the hypercube 
\begin{equation}\label{def:cube}
\text{$\cube(x') = \ballZ(x',R) \subset\subset \Z^d$}. 
\end{equation}
This gives us a partition of $\Z^d$ into disjoint hypercubes. 
\begin{definition}\label{def:edges}
Let $\edges$ be the subset of vertices in $\cube(0)$ such that at least two of their coordinates 
have values in the set $\{-R,-R+1,-R+2,R-2,R-1,R\}$, and let $\edges(x') = x' + \edges$, for all $x'\in \ZZ$. 
We call $\edges(x')$ the {\it frame} of $\cube(x')$. 
\end{definition}
\begin{figure}
\begin{center}
\psfragscanon
\includegraphics[height=5cm]{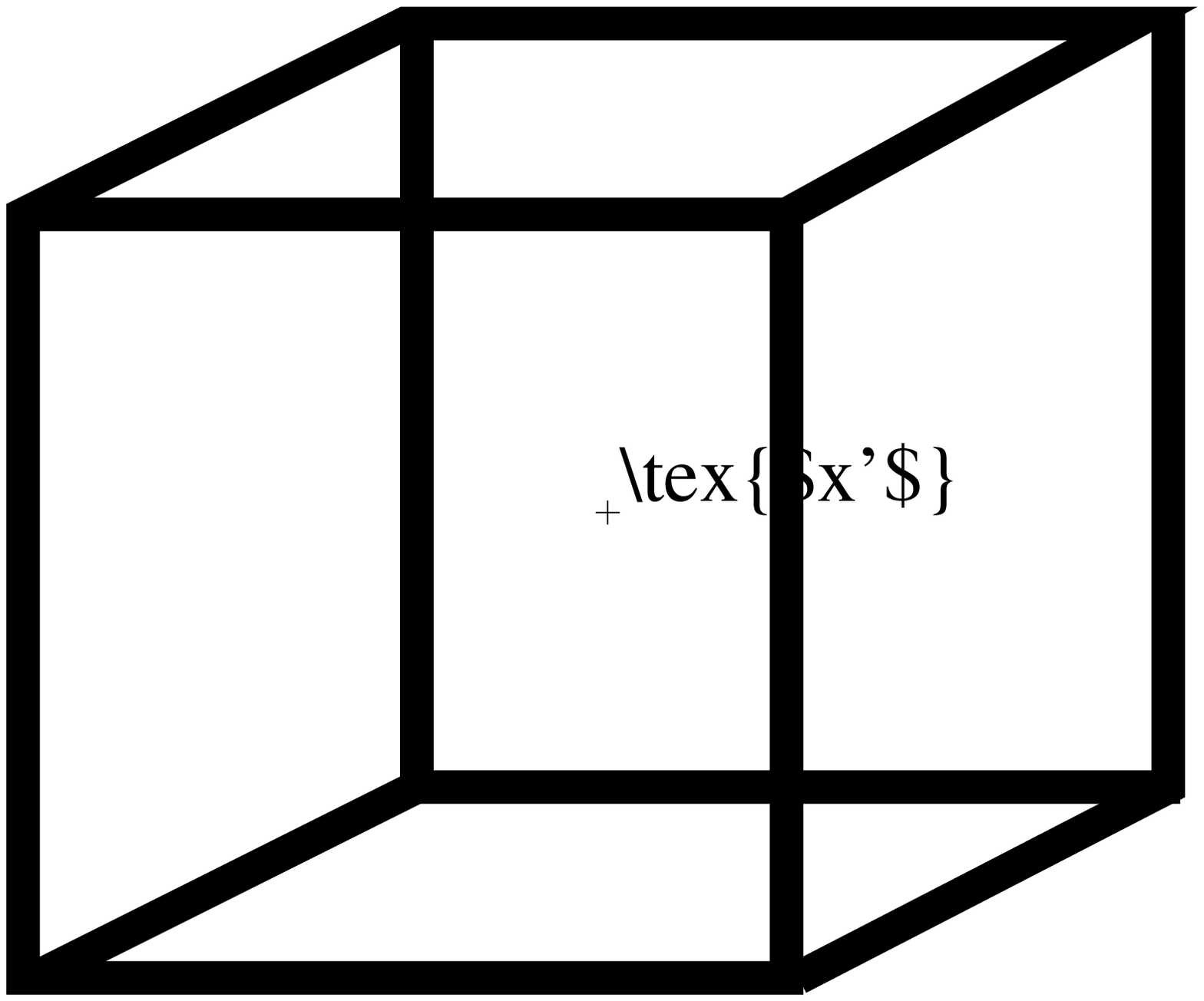}
\caption{The frame of $\cube(x')$ in $\Z^3$.}
\end{center}
\end{figure}
Note that the set $\edges$ is connected in $\Z^d$, and for any $x_1',x_2'\in\ZZ$ nearest-neighbors in $\ZZ$, 
the set $\edges(x_1')\cup\edges(x_2')$ is connected in $\Z^d$. 
 
In the case $d=3$, the set $\cube(x')$ is the usual cube, and the set $\edges(x')$ is just 
the $2$-neighborhood of its edges in the sup-norm, restricted to the vertices inside $\cube(x')$. 

\begin{lemma}\label{l:capaedges}
There exists $C = C(d)<\infty$ such that for all $R \geq 2$, 
\begin{equation}\label{eq:capaedges}
\capa(\edges) \leq C R^{d-2} / \log R ~.\
\end{equation}
\end{lemma}
\begin{proof}[Proof of Lemma~\ref{l:capaedges}]
The proof easily follows from \eqref{Green_decay}, \eqref{def:capacity}, \eqref{eq:hittingprobability:equality}, 
and \eqref{eq:hittingprobability:bounds}.
Let $R\geq 2$. Take $x\in\Z^d$ with $|x| = 2R$. 
Note that for any $y\in\edges$, $R\leq |x-y|\leq 3R$. 
We have 
\[
\capa(\edges)
\stackrel{\eqref{def:capacity}}=
\sum_{y\in\edges} e_\edges(y)
\stackrel{\eqref{Green_decay},\eqref{eq:hittingprobability:equality}}\leq
C R^{d-2} \cdot P_x[H_\edges<\infty] 
\stackrel{\eqref{eq:hittingprobability:bounds}}\leq
C R^{d-2} \cdot 
\sum_{y\in \edges} g(x,y) / \inf_{z\in \edges}\sum_{y\in \edges}g(z,y)
.\
\]
By \eqref{Green_decay}, we get 
\[
\sum_{y\in \edges} g(x,y) \leq C R^{2-d} \cdot |\edges| 
\leq 
C R^{2-d} \cdot \binom{d}{2} \cdot 6^2 \cdot (2R+1)^{d-2} 
\leq C .\
\]
It remains to show that $\inf_{z\in \edges}\sum_{y\in \edges}g(z,y)\geq c\cdot \log R$. 
By the definition of $\edges$, for any $z\in\edges$ and any integer $1\leq k\leq R$, we have 
\[
\left|\{y\in\edges~:~|y - z| = k\}\right| \geq k^{d-3} .\
\]
Therefore, uniformly in $z\in\edges$, we obtain
\[
\sum_{y\in \edges}g(z,y) 
\geq 
\sum_{k=1}^R ~\sum_{y\in\edges~:~|y - z| = k} g(z, y)
\stackrel{\eqref{Green_decay}}\geq 
\sum_{k=1}^R c \cdot k^{2-d}\cdot k^{d-3}
\geq c\cdot \log R .\
\]
Putting all the bounds together we get \eqref{eq:capaedges}. 
\end{proof}

\subsection{Good vertices}

\begin{definition}\label{def:good}
Let $\genl \in \loczd$. We say that $x'\in\ZZ$ is {\it $R$-good} for $\genl$ if 
\begin{enumerate}\itemsep1pt
\item[(1)]
$\genl(x) = 0$ for all $x\in\edges(x')$,
\item[(2)]
$\sum_{x\in\dcube(x')} \genl(x) \leq R^{d-1}$. 
\end{enumerate}
If $x'$ is not $R$-good, then we call it {\it $R$-bad} for $\genl$. 
\end{definition}
\begin{remark}
The choice of $R^{d-1}$ on the right-hand side of (2) is quite arbitrary. 
Any function $f = f(R)$ which grows faster than linearly would serve our purposes (see the proof of Lemma~\ref{l:uRgood}). 
Condition (2) of Definition~\ref{def:good} will be important in Section~\ref{sec:proofthm},
where we use it to give an upper bound on the number of excursions of the interlacement trajectories inside
$\dcube(x')$.
  \end{remark}
Note that for any $R\geq 0$ and $x' \in \ZZ$, 
\begin{equation}\label{good_event_measurable}
\text{ the event $\{\text{$x'$ is $R$-good}\}$ is decreasing and 
$\sigma(\Psi_{y}, y \in \ballZ(x',R))$-measurable.} 
\end{equation}
\begin{lemma}\label{l:uRgood}
For $R\geq 1$, let $u_R = R^{2-d}$. Then 
\begin{equation*}
\mathbb P\left[
\text{$0$ is $R$-good for $\local^{u_R}$}
\right]
\to 1 ,\quad\text{as $R\to\infty$.}
\end{equation*}
\end{lemma}
\begin{proof}[Proof of Lemma~\ref{l:uRgood}]
By the definition of $R$-good vertices, it suffices to prove that 
\[
\mathbb P\left[\edges\subseteq \mathcal V^{u_R}\right]\to 1
\quad \text{ and } \quad
\mathbb P\left[\sum_{x\in\dcube(0)} \local^{u_R}_x \leq R^{d-1}\right] \to 1,\qquad\text{as $R\to\infty$}.
\]
The first statement follows from Lemma~\ref{l:capaedges}. Indeed, 
\[
\mathbb P\left[\edges\subseteq \mathcal V^{u_R}\right] = e^{-u_R\cdot\capa(\edges)}
\stackrel{\eqref{eq:capaedges}}\geq e^{-c/\log R} \to 1 .\
\]
As for the second statement, by the Markov inequality, 
\[
\mathbb P\left[\sum_{x\in\dcube(0)} \local^{u_R}_x > R^{d-1}\right]
\leq 
R^{1-d}\cdot\sum_{x\in\dcube(0)} \mathbb E[\local^{u_R}_x]
=
R^{1-d}\cdot|\dcube(0)|\cdot\mathbb E[\local^{u_R}_0]
\stackrel{\eqref{eq:meanlocal}}\leq C\cdot u_R \to 0 .\
\]
This completes the proof of Lemma~\ref{l:uRgood}. 
\end{proof}

\bigskip

For $V_1,V_2 \subseteq \Z^d$ and $\genl \in \loczd$, we write ``$V_1 \leftrightarrow V_2$ 
 by a $*$-path in $\ZZ$ of $R$-bad vertices for $\genl$'', 
if there is a sequence $\pi=(x_1',\dots,x_n')$  in $\ZZ$ of $R$-bad vertices for $\genl$
such that  
\begin{equation}\label{eq:ZZpath}
x_1' \in V_1, \quad x_n' \in V_2, \qquad \forall \, 1\leq i \leq n-1\, : \, |x_{i+1}'-x_i'|=2R+1. 
\end{equation}
The next lemma proves that $*$-connected components of $R$-bad vertices for $\local^u$ in $\ZZ$ are small for 
large enough $R$ and small enough $u$. 
Then a standard relation between nearest-neighbor and $*$-connectivities implies the existence of 
a unique infinite connected component of $R$-good vertices (see Corollary~\ref{cor:bad:*paths}).

\begin{lemma}\label{l:bad:*paths}
There exist $R\geq 0$, $u_1>0$, $c=c(d)>0$ and $C=c(d)<\infty$ such that for all $u \leq u_1$ and $N\geq 1$, we have   
\begin{equation}\label{eq:connectivity:badvertices:z}
\mathbb P
\big[
\text{$0\leftrightarrow \intbb\ballZZ (0,N)$ by a $*$-path in $\ZZ$ of $R$-bad vertices for $\local^u$}
\big]
\leq C e^{-N^c} .\
\end{equation}
\end{lemma}
\begin{proof}[Proof of Lemma~\ref{l:bad:*paths}]
First of all, note that the $\loczdsig$-measurable event 
\[
\big\{\ell~:~\text{$0\leftrightarrow \intbb\ballZZ (0,N)$ 
by a $*$-path in $\ZZ$ of $R$-bad vertices for $\ell$}\big\}
\]
is increasing. Therefore, 
it suffices to prove that there exist $R\geq 0$, $u>0$, $c>0$ and $C<\infty$ such that for all $N\geq 1$, 
\eqref{eq:connectivity:badvertices:z} holds. 
(Then, by \eqref{eq:stochdomination}, the result will hold for all $u'$ smaller than $u$.)

\medskip

For $x\in\Z^d$ and integers $R\geq 0$, $L\geq 1$, 
consider the events

\begin{equation}\label{def_eq_our_cascading_events}
G_{x,L,R} = 
\begin{cases}
\left\{\genl \in \loczd~:~
\begin{array}{c}
 \ballZ( x,L)  \leftrightarrow \ballZ( x,2L)^c    \\ 
\text{  by a $*$-path in $\ZZ$ of $R$-bad vertices for $\genl$} 
\end{array}
\right\}, \quad &\text{if } L \geq R,  \\
\loczd, \quad &\text{if } L<R.
\end{cases}
\end{equation}

\medskip

In order to prove \eqref{eq:connectivity:badvertices:z}, it suffices to show that 
there exist $L_0 \geq 1$, $l_0 >1$, $R\geq 0$ and $u>0$ such that 
\begin{equation}\label{eq:application:corollary37}
\mathbb P \left[G_{0,L_n,R}(\local^u) \right] \leq 2^{-2^n} ,\quad \text{for all $n\geq 0$} ,\ 
\end{equation}
where $L_n$ are defined in \eqref{def:scalesLn} 
(see also the notation in \eqref{def:Au}).
This will immediately follow from Lemma~\ref{l:probability:cascading}, as soon as we show that 
\begin{equation}\label{eq:application:corollary37:cascading}
\begin{array}{c}
\text{$(G_{x,L,R})_{x\in \Z^d, L\geq 1, R\geq 0 }$ is a family of increasing events}\\
\text{cascading uniformly with complexity at most $d$,}
\end{array}
\end{equation}
and that
the family of events $(G_{x,L,R})_{x\in \Z^d,L\geq 1, R\geq 0 }$ satisfies \eqref{eq:probability:cascading:1}. 

\medskip

We begin with the proof of \eqref{eq:application:corollary37:cascading}. 
The events $G_{x,L,R}$ are clearly increasing. 
For $L\geq R$, we have $\genl \in G_{x,L,R}$ if and only if 
 there exists a $*$-path $\pi' = (y_1',\ldots,y_n')$ in $\ZZ$ of $R$-bad vertices for $\genl$
 satisfying
 \begin{equation}\label{seed_event_iff}
 |y'_1-x| \leq L, \qquad 2L < |y'_n -x|, \qquad \forall \; 1 \leq i \leq n \; : \;
 |y'_i -x| \leq 2L + 2R +1.
 \end{equation}
Treating the cases $L \geq R$ and $L<R$ separately and using 
\eqref{good_event_measurable} and \eqref{seed_event_iff}, one can show that the event $G_{x,L,R}$ is
$\sigma(\Psi_{y}, y \in \ballZ(x,10L))$-measurable.
Let $l$ be a multiple of $100$, $x\in \Z^d$, $R\geq 0$, $L\geq 1$. Let 
\[
\Lambda = L\cdot\Z^d \cap \ballZ(x,3lL) .\
\]
The set $\Lambda$ immediately satisfies \eqref{eq:cascading:2} and \eqref{eq:cascading:3} (with $\lambda=d$), 
so we only need to check that $\Lambda$ satisfies \eqref{eq:cascading:4}. 
By \eqref{def_eq_our_cascading_events}, it is enough to consider the non-trivial
case $L \geq R$.

If $\genl \in G_{x,lL,R}$, then there exists a $*$-path $\pi' = (y_1',\ldots,y_n')$ in $\ZZ$ of $R$-bad
 vertices for $\genl$
 satisfying $|y'_1-x| \leq lL$ and
 $2lL < |y'_n -x|\leq 2lL + 2R +1 \leq 3lL$, so that we can find 
$x_1,x_2 \in \Lambda$ such that $|y'_1 -x_1| \leq L$, $|y'_n -x_2| \leq L$.
Note that $|x_1 - x_2| \geq lL - 2L > \frac{l}{100}L$. 
Moreover, the path $\pi'$ connects 
$\ballZ( x_i,L)$ to $\ballZ( x_i,2L)^c$ for $i\in\{1,2\}$. Thus 
$\genl \in G_{x_1,L,R} \cap G_{x_2,L,R}  $, which implies 
\eqref{eq:cascading:4} and hence \eqref{eq:application:corollary37:cascading}.

\medskip

\noindent
It remains to prove that $(G_{x,L,R})_{x\in \Z^d,L\geq 1, R\geq 0 }$ satisfies \eqref{eq:probability:cascading:1}. 
Let us choose $L_0=R$. By \eqref{seed_event_iff} and \eqref{eq:ZZpath} we have
\begin{equation*}
G_{x,R,R} \subseteq 
\bigcup_{x'\in \ballZ (x,R) \cap \ZZ }
\left\{\genl \in \loczd ~:~\text{$x'$ is $R$-bad for $\genl$}
\right\} .\
\end{equation*}
Since $|\ballZ (x,R) \cap \ZZ | =1$, the condition \eqref{eq:probability:cascading:1}
follows from Lemma~\ref{l:uRgood}. Thus we can apply  Lemma~\ref{l:probability:cascading} to infer
 \eqref{eq:application:corollary37}, which completes
the proof of Lemma~\ref{l:bad:*paths}.
\end{proof}

\bigskip

The following result states that there exists a ubiquitous infinite component of good vertices in $\ZZ$.
It is a consequence of Lemma~\ref{l:bad:*paths} and 
\cite[Lemma~2.23]{Kesten} 
about the connectedness of the exterior $*$-boundary of a $*$-connected subset of $\Z^d$.

\begin{corollary}\label{cor:bad:*paths}
Fix $R$, $u_1$, 
$c=c(d)>0$, and $C=C(d)<\infty$ as in Lemma~\ref{l:bad:*paths}. 
For all $u\leq u_1$, we have 
\begin{enumerate}
\item[(1)]
for all $n,N\geq 1$, 
\begin{equation}\label{eq:good:shell}
\mathbb P\left[
\begin{array}{c}
\text{$\ballZZ(0,N+n)\setminus\ballZZ(0,N)$ contains a set $\mathcal S\subset \ZZ$ such that}\\
\text{$\mathcal S$ is connected in $\ZZ$, each $x\in\mathcal S$ is $R$-good for $\local^u$, and}\\
\text{every $*$-path in $\ZZ$ from $\ballZZ(0,N+1)$ to $\intbb\ballZZ(0,N+n)$}\\
\text{intersects $\mathcal S$}
\end{array}
\right]\geq 
1 - C\cdot|\ballZZ(0,N+1)|\cdot e^{-n^c} ,\
\end{equation}
\item[(2)]
there exists a unique infinite connected component of $R$-good vertices for $\local^u$ in $\ZZ$, 
which we denote by $\mathcal G^\infty$, and for all $n\geq 1$, 
\begin{equation}\label{eq:Ginfty:ball}
\mathbb P\left[
\text{$\mathcal G^\infty$ contains a vertex in $\ballZZ(0,n)$}
\right]
\geq 1 - C\cdot \sum_{N\geq n} e^{-N^c} ,\
\end{equation}
\item[(3)]
the set $\bigcup_{x'\in \mathcal G^\infty}\edges(x')$ is an infinite connected subset of $\mathcal V^u$. 
\end{enumerate}
\end{corollary}
\begin{proof}[Proof of Corollary~\ref{cor:bad:*paths}]
(1) Take $n,N\geq 1$. Let 
\[
\widetilde {\mathcal S} = \ballZZ(0,N)\cup\left\{x\in\ballZZ(0,N+n)~:~
\begin{array}{c}
\text{$x$ is connected to $\ballZZ(0,N+1)$ by a $*$-path}\\
\text{in $\ballZZ(0,N+n)$ of $R$-bad vertices for $\local^u$}
\end{array}
\right\} ,\
\]
and consider the exterior $*$-boundary of $\widetilde {\mathcal S}$ in $\ballZZ(0,N+n)$:
\[
\widehat {\mathcal S} = \left\{ y\in\ballZZ(0,N+n)\setminus\widetilde {\mathcal S}~:~
\text{$y$ is a $*$-neighbor in $\ZZ$ of some $x\in\widetilde {\mathcal S}$}
\right\} .\
\]
Note that every vertex in $\widehat {\mathcal S}$ is $R$-good. 
If $\widetilde {\mathcal S}\cap \intbb\ballZZ(0,N+n) = \emptyset$, then 
every $*$-path in $\ZZ$ from $\ballZZ(0,N+1)$ to $\intbb\ballZZ(0,N+n)$ intersects $\widehat {\mathcal S}$. 
Non-trivially, it was proved in \cite[Lemma~2.23]{Kesten} (see also a short proof in \cite[Theorem~4]{timar_boundary}) that 
if $\widetilde {\mathcal S}\cap \intbb\ballZZ(0,N+n) = \emptyset$, then 
$\widehat {\mathcal S}$ contains a {\it connected} component $\mathcal S$ in $\ZZ$ 
such that every $*$-path in $\ZZ$ from $\ballZZ(0,N+1)$ to $\intbb\ballZZ(0,N+n)$ intersects $\mathcal S$.
By translation invariance of $\local^u$ and \eqref{eq:connectivity:badvertices:z}, 
with $c = c(d)>0$ and $C = C(d)<\infty$ as in Lemma~\ref{l:bad:*paths}, 
and for all $n,N\geq 1$, we have 
\[
\mathbb P\left[
\begin{array}{c}
\text{$\ballZZ(0,N+1)$ is connected to $\intbb\ballZZ(0,N+n)$}\\
\text{by a $*$-path in $\ZZ$ of $R$-bad vertices for $\local^u$}
\end{array}
\right] \leq |\ballZZ(0,N+1)|\cdot C\cdot e^{-n^c} .\
\]
Together with the above observations, this implies the first statement of Corollary~\ref{cor:bad:*paths}. 

\medskip

(2) The existence of $\mathcal G^\infty$ as well as \eqref{eq:Ginfty:ball} follow from \eqref{eq:connectivity:badvertices:z} and 
 planar duality (see, e.g., the proof of \cite[Theorem~2.1]{RS:BP}). 
The uniqueness of $\mathcal G^\infty$ follows from \eqref{eq:good:shell} and the Borel-Cantelli lemma. 

\medskip

(3) The fact that $\cup_{x'\in \mathcal G^\infty}\edges(x')$ is an infinite connected subset of $\mathcal V^u$ follows from (2), Definition~\ref{def:edges} of $\edges$,
and Definition~\ref{def:good} of $R$-good vertices. 
\end{proof}

\bigskip

\section{Conditional independence for random interlacements}\label{sec:condindep}

In this section we prove (in Lemma~\ref{l:condindep}) 
that the behavior of the interlacement trajectories with labels at most $u$
inside a finite set $K$ 
is independent of their behavior outside of $K$, given the information about entrance and exit points 
of all the excursions into $K$ of all the interlacement trajectories with labels at most $u$. 
As part of the proof, we will also identify the conditional law of the excursions inside and outside $K$ 
(see \eqref{eq:EinnK:proba} and \eqref{eq:EoutK:proba}, respectively). 

We begin by introducing notation and recalling some properties of the interlacement point measures, 
which we will use to identify the above mentioned laws of excursions. 
We then properly define the excursions (in Section~\ref{sec:condindep:2})  
and the $\sigma$-algebras of events generated by excursions inside, outside, and on the boundary of $K$ 
(in Section~\ref{sec:condindep:3}). 
Finally, (in Section~\ref{sec:condindep:4}) we state and prove the conditional independence of 
the $\sigma$-algebras.

\subsection{More preliminaries about interlacements}\label{sec:condindep:1}

Recall the notation and the definition of the interlacement point process from Section~\ref{sec:randominterlacements}. 
Let $\omega = \sum_{i\geq 0}\delta_{(w^*_i,u_i)}$ be an interlacement point process on $W^*\times\R_+$. 
For $K\subset\subset\Z^d$ and $u>0$, let 
\begin{equation}\label{def:omegaKu:omega-omegaKu}
\omega_{K,u} = \sum_{i\geq 0}\delta_{(w^*_i,u_i)}\mathds{1}_{\{w^*_i\in W^*_K,\, u_i\leq u\}} 
\quad\mbox{and}\quad
\omega - \omega_{K,u} = \sum_{i\geq 0}\delta_{(w^*_i,u_i)}\mathds{1}_{\{w^*_i\notin W^*_K\}\cup\{u_i >  u\}} 
\end{equation}
be the restrictions of $\omega$ to the set of pairs $(w^*_i,u_i)$ with, respectively,  
$w^*_i$ intersecting $K$ and $u_i\leq u$, and   
either $w^*_i$ not intersecting $K$ or $u_i> u$. 
By the definition of $\omega$, the point measures $\omega_{K,u}$ and $\omega - \omega_{K,u}$ are independent Poisson point processes. 
By \eqref{eq:omega}, each $\omega_{K,u}$ is a finite point measure. 
For each $K\subset\subset\Z^d$ and $u>0$, 
$\omega_{K,u}$ is a Poisson point process on $W^*_K\times \R_+$ with intensity measure
\[
\mathds{1}_{W^*_K\times [0,u]}\cdot\nu(\mathrm{d} w^*) \mathrm{d}u ,\
\]
where the measure $\nu$ is defined in \eqref{def_eq_nu}. 
In particular, the total mass of $\omega_{K,u}$ has Poisson distribution with parameter $u\cdot \capa(K)$ 
(this follows from \eqref{def_Q_K} and \eqref{def_eq_nu}), and 
all the $u_i$'s in the definition of $\omega_{K,u}$ are almost surely different. 
Therefore, $\omega_{K,u}$ admits the following representation: 
\begin{equation}\label{eq:omegaKu}
\omega_{K,u} = \sum_{i=1}^{N_{K,u}}\delta_{(w^*_i,u_i)} ,
\end{equation}
where $N_{K,u}$ has  Poisson distribution with parameter $u\cdot \capa(K)$, and 
given $N_{K,u}$, (a) $(u_1,\ldots,u_{N_{K,u}})$ and $(w^*_1,\ldots,w^*_{N_K,u})$ are independent, 
(b) $u_1<\ldots<u_{N_{K,u}}$ are obtained by relabeling independent uniform random variables on $[0,u]$, 
(c) $w^*_i$ are independent and each distributed according to $\mathds{1}_{W^*_K}\cdot\nu(\mathrm{d} w^*)/\capa(K)$.

For each $w^*_i$ in \eqref{eq:omegaKu}, 
\begin{equation}\label{def:Xi}
\begin{array}{c}
\text{let $X_i$ be the unique trajectory from $(\pi^*)^{-1}(w^*_i)\subset W$ parametrized in such a way that}\\
\text{$X_i(0)\in K$ and $X_i(t)\notin K$ for all $t<0$.}
\end{array}
\end{equation} 
(Here we abuse notation and denote by $X_i$ (bi-infinite) trajectories rather than canonical 
coordinate maps in $W$ or $W_+$, see below \eqref{def_eq_W_plus}.)
By \eqref{def_Q_K} and \eqref{def_eq_nu}, 
given $N_{K,u}$ and $(u_i~:~1\leq i\leq N_{K,u})$, the random trajectories
 $(X_i~:~1\leq i\leq N_{K,u})$  are independent and for all $\mathcal A,\mathcal B\in\mathcal W_+$ (see below \eqref{def_eq_W_plus}), $x\in \Z^d$, 
\begin{equation}\label{eq:QK}
\mathbb P[(X_i(-t)~:~t\geq 0)\in \mathcal A,~ X_i(0)=x,~ (X_i(t)~:~t\geq 0)\in \mathcal B)]
=
P^K_x[\mathcal A]\cdot \widetilde e_K(x) \cdot P_x[\mathcal B] ,\
\end{equation}
where $P^K_x$ is the law of simple random walk started at $x$ and conditioned on $\widetilde H_K = \infty$, 
and $\widetilde e_K$ is defined in \eqref{def:normalizedequilibriummeasure}.

\subsection{Interlacement excursions}\label{sec:condindep:2}

\begin{definition}\label{def:excursions}
For $w\in W$, 
let $R_1(w) = \inf\{n\in\Z~:~w(n)\in K\}$ be the first entrance time of $w$ to $K$. 
If $R_1(w) < \infty$, let $D_1(w) =\inf\{n > R_1(w)~:~w(n)\notin K\}$ be the first exit time from $K$.
Similarly, for $k\geq 2$, if $R_{k-1}(w)<\infty$, let 
\[
D_{k-1}(w) = \inf\{n > R_{k-1}(w)~:~w(n)\notin K\} \quad\mbox{and}\quad
R_k(w) = \inf\{n > D_{k-1}(w)~:~w(n)\in K\} .\
\]
For $w$ with $R_1(w)<\infty$, let 
\[
M(w) = \max\{k\geq 1~:~ R_k(w) < \infty\} .\ 
\]
By \eqref{def:W}, $M(w)<\infty$ for any $w\in W$. 

Abusing notation, we extend the above definitions of $R_k$, $D_k$ and $M$ to trajectories $w_+\in W_+$ in a natural way, namely, 
defining $R_1(w_+) = H_K(w_+)$ (see \eqref{def:HU}), and all the other variables 
with the same formulas as above. 
\end{definition}

\bigskip

Given $(X_i~:~1\leq i\leq N_{K,u})$ as in \eqref{def:Xi}, 
for each $1\leq i\leq N_{K,u}$, let 
\[
M_i = M(X_i)
\]
be the number of times trajectory $X_i$ revisits $K$, 
and for each $1\leq j\leq M_i$, let 
\[
A_{i,j} = R_j(X_i) \quad \mbox{and}\quad
B_{i,j} = D_j(X_i) - 1 
\]
be the times when $j$th excursion of $X_i$ inside $K$ begins and ends. 
Note that $X_i(t)\in K$ if and only if $A_{i,j}\leq t\leq B_{i,j}$ for some $1\leq j\leq M_i$. 
For $1\leq i\leq N_{K,u}$ and $1\leq j\leq M_i$, let 
\[
T^\inn_{i,j} = B_{i,j} - A_{i,j}, \quad 
X^{\inn}_{i,j}(t) = X_i(A_{i,j} + t), \quad\text{for $0\leq t\leq T^\inn_{i,j}$} ,\
\]
and for $1\leq i\leq N_{K,u}$, $1\leq j\leq M_i-1$, let 
\[
T^\outt_{i,j} = A_{i,j+1} - B_{i,j}, \quad
X^\outt_{i,j}(t) = X_i(B_{i,j} + t), \quad\text{for $0\leq t\leq T^\outt_{i,j}$.}
\]
Note that $(X^\inn_{i,j}~:~1\leq j\leq M_i)$ correspond to the pieces of $X_i$ inside $K$, and 
$(X^\outt_{i,j}~:~1\leq j\leq M_i-1)$ correspond to the finite pieces of $X_i$ outside $K$ 
(except for their start and end points). 
Finally, let 
\[
X^-_i(t) = X_i(-t)\quad \mbox{and}\quad
X^+_i(t) = X_i(t+B_{i,M_i}),\quad\mbox{for }t\geq 0 ,\
\]
be the (infinite) pieces of trajectory $X_i$ up to the first enter in $K$ and from the last visit to $K$, respectively.

\subsection{Interior, exterior, and boundary $\sigma$-algebras}\label{sec:condindep:3}

Let $\mathcal F^\inn_{K,u}$ be the $\sigma$-algebra generated by the random variables 
\[
N_{K,u},\quad \left(u_i~:~1\leq i\leq N_{K,u}\right),\quad 
\left(M_i~:~1\leq i\leq N_{K,u}\right),\quad
\left(X^\inn_{i,j}~:~1\leq i\leq N_{K,u},~ 1\leq j\leq M_i\right) ,\
\]
i.e., $\mathcal F^\inn_{K,u}$ is generated by the excursions of the interlacement trajectories 
with labels at most $u$ inside $K$. 

Let $\mathcal F^\outt_{K,u}$ be the $\sigma$-algebra generated by 
\begin{multline*}
\omega - \omega_{K,u}, \quad N_{K,u},\quad \left(u_i~:~1\leq i\leq N_{K,u}\right),\quad 
\left(M_i~:~1\leq i\leq N_{K,u}\right),\\
\left(X^-_i~:~1\leq i\leq N_{K,u}\right),\quad
\left(X^\outt_{i,j}~:~1\leq i\leq N_{K,u},~ 1\leq j\leq M_i-1\right), \quad
\left(X^+_i~:~1\leq i\leq N_{K,u}\right)
\end{multline*}
i.e., $\mathcal F^\outt_{K,u}$ is generated by the excursions of the interlacement trajectories 
with labels at most $u$ outside $K$ and $\omega - \omega_{K,u}$ 
(see \eqref{def:omegaKu:omega-omegaKu}). 

Let $\mathcal F^{AB}_{K,u}$ be the $\sigma$-algebra generated by 
\begin{multline*}
N_{K,u},\quad \left(u_i~:~1\leq i\leq N_{K,u}\right),\quad 
\left(M_i~:~1\leq i\leq N_{K,u}\right),\\
\left((X_i(A_{i,j}), X_i(B_{i,j}))~:~1\leq i\leq N_{K,u},~ 1\leq j\leq M_i\right) ,\
\end{multline*}
i.e., $\mathcal F^{AB}_{K,u}$ is generated by the entrance and exit points 
of the interlacement trajectories with labels at most $u$ to $K$.

The following properties are immediate from the definitions. 
\begin{claim}\label{claim:sigmaF}
For any $K\subset\subset \Z^d$, 
\begin{itemize}
\item[(1)]
$\mathcal F^{AB}_{K,u}\subset\mathcal F^\inn_{K,u}$ and 
$\mathcal F^{AB}_{K,u}\subset\mathcal F^\outt_{K,u}$, 
\item[(2)]
$\sigma(\mathcal F^\inn_{K,u},\mathcal F^\outt_{K,u}) = \salgebra$ (see below \eqref{eq:omega}), 
\item[(3)]
$(\local^u_x~:~x\in K)$ is $\mathcal F^\inn_{K,u}$-measurable, and 
$(\local^u_x~:~x\in\Z^d\setminus K)$ is $\mathcal F^\outt_{K,u}$-measurable 
(recall the definition of $\local^u$ in \eqref{def_eq_local_time_at_level_u}). 
\end{itemize}
\end{claim}

\subsection{Conditional independence}\label{sec:condindep:4}

In this section we prove the main result of Section~\ref{sec:condindep}, which states that the
$\sigma$-algebras $\mathcal F^\inn_{K,u}$ 
(generated by the excursions of the interlacement trajectories inside $K$)
 and $\mathcal F^\outt_{K,u}$ (generated by the excursions outside $K$ and $\omega - \omega_{K,u}$ 
(see \eqref{def:omegaKu:omega-omegaKu}))
 are conditionally independent, 
given $\mathcal F^{AB}_{K,u}$ (generated by the entrance and exit points 
of the interlacement trajectories to $K$).
 In the proof of \eqref{eq:sts:inerlacement:2}, we will only use 
Lemma~\ref{l:condindep}(a) and \eqref{eq:EinnK:proba} 
(see the proofs of Lemmas~\ref{l:condindep:good} and \ref{l:Pk}, respectively).
We begin with a definition. 

\begin{definition}\label{def:eventsE}
For integers $n\geq 1$, $1\leq i\leq n$, $\mathcal U_i\in\mathcal B([0,u])$, 
$\mathcal A_i,\mathcal B_i\in\mathcal W_+$, integers $m_i\geq 1$, 
$1\leq j\leq m_i$, $x_{i,j},y_{i,j}\in \intb K$, finite nearest-neighbor trajectories 
$\tau^\inn_{i,j}$ from $x_{i,j}$ to $y_{i,j}$ in $K$, and 
for $1\leq j'\leq m_i-1$, finite nearest-neighbor trajectories 
$\tau^\outt_{i,j'}$ from $y_{i,j'}$ to $x_{i,j'+1}$ outside $K$ except for the start and end points, 
consider the events 
\begin{equation}\label{def:EABK}
\mathcal E^{AB}_{K,u} 
= 
\left\{
\begin{array}{c}
N_{K,u}=n,~ u_i\in\mathcal U_i,~ M_i= m_i,~ X_i(A_{i,j}) = x_{i,j},~ X_i(B_{i,j}) = y_{i,j},\\
\text{for all $1\leq i\leq n$, $1\leq j\leq m_i$}
\end{array}
\right\}
\in \mathcal F^{AB}_{K,u} ,\
\end{equation}
\begin{equation}\label{def:EinK}
\mathcal E^\inn_{K,u}
=
\left\{
\begin{array}{c}
N_{K,u}=n,~ u_i\in\mathcal U_i,~ M_i= m_i,~ X_i(A_{i,j}) = x_{i,j},~ X_i(B_{i,j}) = y_{i,j},\\
X^\inn_{i,j} = \tau^\inn_{i,j},\quad
\text{for all $1\leq i\leq n$, $1\leq j\leq m_i$}
\end{array}
\right\}
\in \mathcal F^\inn_{K,u} ,\
\end{equation}
\begin{equation}\label{def:EoutK}
\mathcal E^\outt_{K,u}
=
\left\{
\begin{array}{c}
N_{K,u}=n,~ u_i\in\mathcal U_i,~ M_i= m_i,~ X_i(A_{i,j}) = x_{i,j},~ X_i(B_{i,j}) = y_{i,j},\\
X^\outt_{i,j'}= \tau^\outt_{i,j'},~ X^-_i\in\mathcal A_i,~ X^+_i \in \mathcal B_i,\\ 
\text{for all $1\leq i\leq n$, $1\leq j\leq m_i$, $1\leq j'\leq m_i-1$}
\end{array}
\right\}
\in \mathcal F^\outt_{K,u} .\
\end{equation}
\end{definition}
Note that 
\begin{equation}\label{rem:sigmagen}
\begin{array}{c}
\text{the $\sigma$-algebras $\mathcal F^\inn_{K,u}$ and $\mathcal F^{AB}_{K,u}$ are respectively generated 
by events of form \eqref{def:EinK} and \eqref{def:EABK},}\\
\text{and $\mathcal F^\outt_{K,u}$ is generated by the events $\mathcal E^\outt_{K,u}\cap \{\omega - \omega_{K,u}\in\mathcal E\}$,
with $\mathcal E\in \salgebra$ (see below \eqref{eq:omega}).}
\end{array}
\end{equation}

\begin{lemma}\label{l:condindep}
For any $K\subset\subset \Z^d$, $u>0$, \\
(a) $\mathcal F^\inn_{K,u}$ and $\mathcal F^\outt_{K,u}$ are conditionally independent, given $\mathcal F^{AB}_{K,u}$, and\\ 
(b) For any choice of the parameters in Definition~\ref{def:eventsE}, 
we have 
\begin{equation}\label{eq:condindep}
\mathbb P\left[\mathcal E^\inn_{K,u}\cap\mathcal E^\outt_{K,u}\right]
=
\mathbb P\left[\mathcal E^{AB}_{K,u}\right]\cdot
\mathbb P\left[\mathcal E^\inn_{K,u}~|~\mathcal E^{AB}_{K,u}\right]\cdot
\mathbb P\left[\mathcal E^\outt_{K,u}~|~\mathcal E^{AB}_{K,u}\right] ,\
\end{equation}
and
\begin{align}
\mathbb P[\mathcal E^{AB}_{K,u}]
&= \mathbb P[N_{K,u} = n]\cdot
\mathbb P[u_i\in\mathcal U_i~:~1\leq i\leq n]\label{eq:EABK:proba} \\ 
&\quad\quad\prod_{i=1}^n\widetilde e_K(x_{i,1})\cdot 
P_{x_{i,1}}\left[
\begin{array}{c}
M(X) = m_i,~X(R_j) = x_{i,j},~X(D_j-1) = y_{i,j}\\
\text{for all $1\leq j\leq m_i$} 
\end{array}
\right] ,\nonumber \\
\mathbb P[\mathcal E^\inn_{K,u}~|~\mathcal E^{AB}_{K,u}]
&=
\prod_{i=1}^n\prod_{j=1}^{m_i}
\frac{P_{x_{i,j}}[(X(t)~:~0\leq t\leq T_K-1)= \tau^\inn_{i,j}]}
{P_{x_{i,j}}[X(T_K-1)= y_{i,j}]},\label{eq:EinnK:proba} \\
\mathbb P[\mathcal E^\outt_{K,u}~|~\mathcal E^{AB}_{K,u}]
&=
\prod_{i=1}^n P^K_{x_{i,1}}[\mathcal A_i]\cdot P^K_{y_{i,m_i}}[\mathcal B_i]\cdot
\prod_{j'=1}^{m_i-1}
\frac{P_{y_{i,j'}}[(X(t)~:~0\leq t\leq \widetilde H_K)= \tau^\outt_{i,j'},~\widetilde H_K<\infty]}
{P_{y_{i,j'}}[X(\widetilde H_K)= x_{i,j'+1},~\widetilde H_K<\infty]} ,\label{eq:EoutK:proba} 
\end{align}
where $T_K$ and $\widetilde H_K$ are defined in \eqref{def:TU} and \eqref{def:tildeHU}, respectively. 
\end{lemma}
\begin{proof}[Proof of Lemma~\ref{l:condindep}]
Statement (a) immediately follows from \eqref{eq:condindep}, 
the fact that point processes $\omega_{K,u}$ and $\omega-\omega_{K,u}$ are independent, 
the inclusion $\mathcal E^\inn_{K,u},~\mathcal E^\outt_{K,u}\subseteq\mathcal E^{AB}_{K,u}$, 
and \eqref{rem:sigmagen}.

\medskip

To prove (b), we first observe that the expressions in \eqref{eq:EABK:proba}, \eqref{eq:EinnK:proba}, and \eqref{eq:EoutK:proba} 
indeed give rise to probability distributions. 

We rewrite the left-hand side of \eqref{eq:condindep} using the definition \eqref{eq:omegaKu} of $\omega_{K,u}$ and \eqref{eq:QK} as 
\begin{multline*}
\mathbb P[\mathcal E^\inn_{K,u}\cap\mathcal E^\outt_{K,u}] ~=~
\mathbb P[N_{K,u}=n]\cdot
\mathbb P[u_i\in\mathcal U_i~:~1\leq i\leq n]\cdot
\prod_{i=1}^nP^K_{x_{i,1}}[\mathcal A_i]\cdot\widetilde e_K(x_{i,1})\\
\cdot
\prod_{i=1}^n
P_{x_{i,1}}\left[
\begin{array}{c}
M(X)=m_i, \quad X(R_j) = x_{i,j},\quad X(D_j-1) = y_{i,j},\\
(X(t)~:~R_j\leq t\leq D_j-1) = \tau^\inn_{i,j},\quad 
(X(t)~:~D_{j'}-1\leq t\leq R_{j'+1})= \tau^\outt_{i,j'},\\
(X(t+D_{m_i}-1)~:~t\geq 0) \in \mathcal B_i,~~~
\text{for all $1\leq j\leq m_i$, $1\leq j'\leq m_i-1$}
\end{array}
\right] .\
\end{multline*}
Note that this equality immediately implies \eqref{eq:EABK:proba} by taking all $\mathcal A_i$ and $\mathcal B_i$ equal to $W_+$ 
and summing over all possible paths $\tau^\inn_{i,j}$ and $\tau^\outt_{i,j'}$. 
 
Consecutive applications of the Markov property for simple random walk 
imply that 
the above expression equals
\begin{multline}\label{eq:condindep:rw}
\mathbb P[N_{K,u}=n]\cdot
\mathbb P[u_i\in\mathcal U_i~:~1\leq i\leq n]\cdot
\prod_{i=1}^nP^K_{x_{i,1}}[\mathcal A_i]\cdot\widetilde e_K(x_{i,1})\\
\cdot\prod_{i=1}^n
\prod_{j=1}^{m_i}
P_{x_{i,j}}\left[X(t) = \tau^\inn_{i,j}(t)~:~ 0\leq t\leq |\tau^\inn_{i,j}|-1\right]\\
\cdot
\prod_{i=1}^n
\prod_{j'=1}^{m_i-1}
P_{y_{i,j'}}\left[X(t) = \tau^\outt_{i,j'}(t)~:~ 0\leq t\leq |\tau^\outt_{i,j'}|-1\right]\\
\cdot\prod_{i=1}^n P_{y_{i,m_i}}\left[\mathcal B_i, \widetilde H_K = \infty\right] .\
\end{multline}
We will now rearrange the terms in \eqref{eq:condindep:rw} to obtain \eqref{eq:condindep}, \eqref{eq:EinnK:proba}, 
and \eqref{eq:EoutK:proba}. 
We begin with a few observations. 
Note that 
\begin{equation}\label{eq:condindep:rw:0}
P_{y_{i,j'}}\left[X(t) = \tau^\outt_{i,j'}(t),\quad 0\leq t\leq |\tau^\outt_{i,j'}|-1\right]
=
P_{y_{i,j'}}\left[(X(t)~:~0\leq t\leq \widetilde H_K) = \tau^\outt_{i,j'}\right] ,\ 
\end{equation}
and
\begin{equation}\label{eq:condindep:rw:1}
P_{y_{i,m_i}}\left[\mathcal B_i, \widetilde H_K = \infty\right]
=
e_K(y_{i,m_i})\cdot P^K_{y_{i,m_i}}[\mathcal B_i] .\
\end{equation}
Also note that by the Markov property at time $|\tau^\inn_{i,j}|-1$, we have 
\begin{multline}\label{eq:condindep:rw:2}
P_{x_{i,j}}\left[X(t) = \tau^\inn_{i,j}(t),\quad 0\leq t\leq |\tau^\inn_{i,j}|-1\right]\\
=
\frac{P_{x_{i,j}}\left[X(t) = \tau^\inn_{i,j}(t),\quad 0\leq t\leq |\tau^\inn_{i,j}|-1,\quad X(|\tau^\inn_{i,j}|)\notin K\right]}
{P_{y_{i,j}}\left[X(1)\notin K\right]} \\
=
\frac{P_{x_{i,j}}\left[(X(t)~:~0\leq t\leq T_K-1) = \tau^\inn_{i,j}\right]}
{P_{y_{i,j}}\left[X(1)\notin K\right]}. 
\end{multline}
We now plug in the expressions \eqref{eq:condindep:rw:0}, \eqref{eq:condindep:rw:1}, and \eqref{eq:condindep:rw:2} 
into \eqref{eq:condindep:rw} to get that 
$\mathbb P[\mathcal E^\inn_{K,u}\cap\mathcal E^\outt_{K,u}]$ equals
\begin{equation}\label{eq:condindep:rw:detailed}
\begin{split}
&\mathbb P[N_{K,u}=n]\cdot
\mathbb P[u_i\in\mathcal U_i~:~1\leq i\leq n]\cdot
\prod_{i=1}^nP^K_{x_{i,1}}[\mathcal A_i]\cdot\widetilde e_K(x_{i,1})\cdot 
e_K(y_{i,m_i})\cdot P^K_{y_{i,m_i}}[\mathcal B_i]\\
&\prod_{i=1}^n
\prod_{j=1}^{m_i}
\frac{P_{x_{i,j}}\left[(X(t)~:~0\leq t\leq T_K-1) = \tau^\inn_{i,j}\right]}
{P_{y_{i,j}}\left[X(1)\notin K\right]}\cdot
\prod_{j'=1}^{m_i-1}
P_{y_{i,j'}}\left[
(X(t)~:~0\leq t\leq \widetilde H_K) = \tau^\outt_{i,j'}
\right] .\
\end{split}
\end{equation}
By taking $\mathcal A_i = \mathcal B_i = W_+$ in \eqref{eq:condindep:rw:detailed} and summing over all $\tau^\inn_{i,j}$ 
and $\tau^\outt_{i,j'}$, we obtain that 
\begin{multline}\label{eq:EABK:proba:2}
\mathbb P[\mathcal E^{AB}_{K,u}]
= \mathbb P[N_{K,u} = n]\cdot
\mathbb P[u_i\in\mathcal U_i~:~1\leq i\leq n] \\ 
\cdot\prod_{i=1}^n\widetilde e_K(x_{i,1})\cdot 
e_K(y_{i,m_i})\cdot
\prod_{j=1}^{m_i}\frac{P_{x_{i,j}}[X(T_K-1) = y_{i,j}]}{P_{y_{i,j}}[X(1)\notin K]}\cdot
\prod_{j'=1}^{m_i-1}P_{y_{i,j'}}[X(\widetilde H_K) = x_{i,j'+1}] .
\end{multline}
The expression \eqref{eq:EinnK:proba} follows from \eqref{eq:condindep:rw:detailed} by 
taking all $\mathcal A_i=\mathcal B_i=W_+$ in \eqref{eq:condindep:rw:detailed}, summing over all $\tau^\outt_{i,j'}$, and 
dividing by \eqref{eq:EABK:proba:2}. 
Similarly, the expression \eqref{eq:EoutK:proba} follows from \eqref{eq:condindep:rw:detailed} by 
summing \eqref{eq:condindep:rw:detailed} over all $\tau^\inn_{i,j}$ and dividing by \eqref{eq:EABK:proba:2}. 

Finally, to obtain \eqref{eq:condindep}, we observe that 
the product of the right-hand sides of \eqref{eq:EinnK:proba}, \eqref{eq:EoutK:proba}, and \eqref{eq:EABK:proba:2} 
equals \eqref{eq:condindep:rw:detailed}. 
The proof of Lemma~\ref{l:condindep} is complete. 
\end{proof}

\bigskip

\section{Proof of Theorem~\ref{thm:sts:interlacement}}\label{sec:proofthm}

Statement \eqref{eq:sts:interlacement:1} of Theorem~\ref{thm:sts:interlacement} follows from Corollary~\ref{cor:bad:*paths} (2) and (3). 
Statement \eqref{eq:sts:inerlacement:2} is proved in Section~\ref{sec:proof:sts:int:2}. 
We will deduce it there from Claim~\ref{claim:HN:proba} and Lemma~\ref{l:eventA:estimate}, which we state in Section~\ref{sec:ubiq}. 

We begin with a general overview of the proof of \eqref{eq:sts:inerlacement:2}. 
As we already know from Corollary~\ref{cor:bad:*paths}, 
we can choose $R$ and $u$ such that $\mathcal V^u$ contains 
an infinite connected subset $\cup_{x'\in \mathcal G^\infty}\edges(x')$, 
where $\mathcal G^\infty$ is the unique infinite connected component of $R$-good vertices in $\ZZ$ for $\mathcal L^u$. 
The goal is to show that if a vertex of $\Z^d$ is in a large connected component of $\mathcal V^u$, then, with high probability, 
it must be (locally) connected to $\cup_{x'\in \mathcal G^\infty}\edges(x')$. 
This is realized in Lemma~\ref{l:eventA:estimate}. 
The crucial observation is that by Corollary~\ref{cor:bad:*paths}, 
with high probability, any long nearest-neighbor path in $\Z^d$ will 
often intersect $\cup_{x'\in \mathcal G^\infty}\ballZ(x',R)$ (see \eqref{eq:HN:proba} and \eqref{eq:Skproperty}). 

The proof of Lemma~\ref{l:eventA:estimate} proceeds by exploring the connected component of a vertex in $\mathcal V^u$,
and showing that every visit to a new box of $\cup_{x'\in \mathcal G^\infty}\ballZ(x',R)$
gives a fresh, uniformly positive chance for the (already explored) vacant set to merge with 
$\cup_{x'\in \mathcal G^\infty}\edges(x')$ (see Lemmas \ref{l:eventA:estimate:k} and \ref{l:Pkbad}). 
The key observation in 
proving that the history of this exploration does not have a negative effect
on the success probability of the next merger 
comes from Lemma~\ref{l:condindep}:  
if we consider a box of radius $R$, the events
which depend on the behavior of the interlacement trajectories outside this box are conditionally
independent of what they do inside the box, given the collection of entrance and exit points
of the excursions inside the box. 
As we already pointed out earlier, some care is still needed, since 
random interlacements do not posess the finite energy property. 
Our definition of good vertices (more precisely, property (1) of Definition~\ref{def:good}) 
allows to overcome this difficulty (see the proof of Lemma~\ref{l:Pkbad}). 
In order to get a uniform lower bound in \eqref{eq:Pkbad:1} of Lemma~\ref{l:Pkbad}, 
we use the fact that the number of excursions of the interlacement trajectories inside
good boxes (corresponding to good vertices) is bounded (see property (2) of Definition~\ref{def:good}). 

\medskip

We now proceed with the proof of \eqref{eq:sts:inerlacement:2}. 
\begin{equation}\label{eq:Ru_1}
\text{From now on we fix $R$ and $u_1$ that satisfy \eqref{eq:connectivity:badvertices:z}, and consider $u\leq u_1$.}
\end{equation}
Since $R$ is now fixed, we will call $R$-good/$R$-bad vertices (see Definition~\ref{def:good}) simply good/bad.

\subsection{Large cluster in $\mathcal V^u$ is likely to be ubiquitous}\label{sec:ubiq}

\noindent
The main result of this section is Lemma~\ref{l:eventA:estimate}. 
We begin with definitions and preliminary observations. 
Recall the definitions of the coarse grained lattice $\ZZ$ from \eqref{eq:ZZ} and 
the ball $\ballZZ(x',N)$ in $\ZZ$ from below \eqref{eq:ZZ}. 

For $N\geq 1$, let 
\begin{equation*}
k_N = \lfloor \sqrt N \rfloor \;\; \text{ and } \;\;
 K_{N,k} = N + k_N\cdot k, \;\; \text{ for } \;\; 0\leq k\leq k_N. 
\end{equation*}

Now we define an event that a large hypercube $\ballZZ(0,2N)$ in $\ZZ$ contains 
a (large) connected component of good vertices in $\ZZ$ which contains separating shells in 
each of $k_N$ concentric annuli $\ballZZ(0,K_{N,k})\setminus\ballZZ(0,K_{N,k-1})$, $1\leq k\leq k_N$.

\begin{definition}\label{def:HN}
For $N\geq 1$, let $\mathcal H_N$ be the event that 
\begin{enumerate}\itemsep1pt
\item
$\ballZZ(0,N)$ is connected to $\intbb\ballZZ(0,2N)$ by a nearest-neighbor path of good vertices for $\local^u$ in $\ZZ$, 
\item
for all $1\leq k\leq k_N$, 
$\ballZZ(0,K_{N,k})\setminus\ballZZ(0,K_{N,k-1})$ contains a set $\mathcal S_k\subset \ZZ$ 
(which we call a shell in $\ballZZ(0,K_{N,k})$ around $\ballZZ(0,K_{N,k-1})$) such that
\begin{itemize}
\item[(a)]
$\mathcal S_k$ is connected in $\ZZ$, 
\item[(b)]
each $x\in\mathcal S_k$ is good for $\local^u$, and
\item[(c)]
every $*$-path in $\ZZ$ from $\ballZZ(0,K_{N,k-1})$ to $\intbb\ballZZ(0,K_{N,k})$ intersects $\mathcal S_k$.
\end{itemize}
\end{enumerate}
\end{definition}
\begin{figure}
\begin{center}
\psfragscanon
\includegraphics[height=8cm]{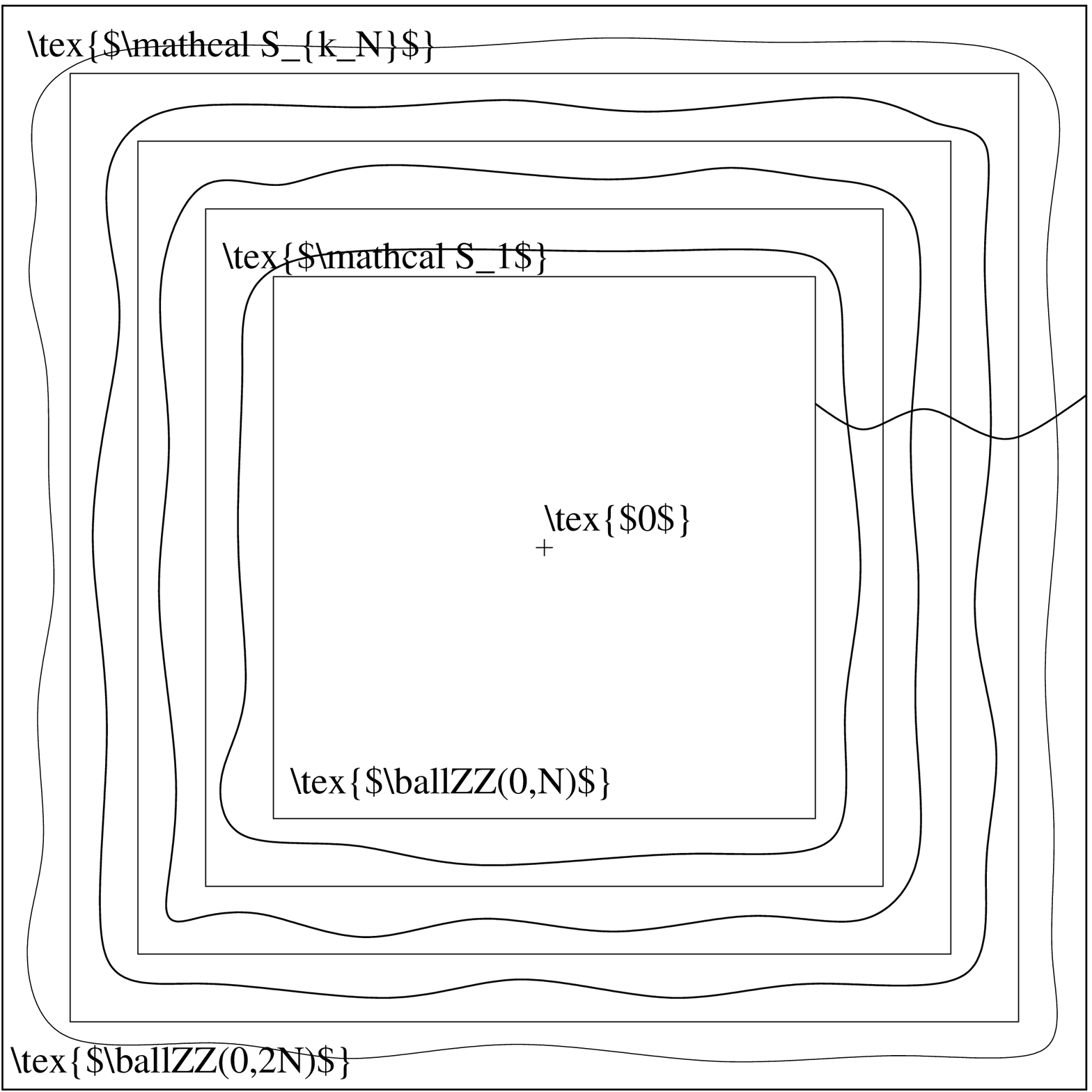}
\caption{The event $\mathcal H_N$. In each of the $k_N$ concentric annuli 
$\ballZZ(0,K_{N,k})\setminus\ballZZ(0,K_{N,k-1})$, $1\leq k\leq k_N$, there exists a connected component 
$\mathcal S_k$ in $\ZZ$ of good vertices (which we call a shell) 
separating $\ballZZ(0,K_{N,k-1})$ from $\intbb\ballZZ(0,K_{N,k})$, and 
all the $\mathcal S_k$ are (disjoint) parts of the same connected component of good vertices in $\ballZZ(0,2N)$.}
\end{center}
\end{figure}
\begin{claim}\label{claim:HN:proba}
It follows from Corollary~\ref{cor:bad:*paths} that for $R$ and $u \leq u_1$ as in \eqref{eq:Ru_1}, 
there exist constants $c=c(d)>0$ and $C=C(d)<\infty$ (possibly different from the ones in Lemma~\ref{l:bad:*paths}) such that 
\begin{equation}\label{eq:HN:proba}
\mathbb P[\mathcal H_N] \geq 1 - Ce^{-N^c} .\
\end{equation}
\end{claim}
Note that if $\mathcal H_N$ occurs, then 
for each $1\leq k\leq k_N$, 
\begin{equation}\label{eq:Sk}
\begin{array}{c}
\text{$\mathcal S_k$ can be defined as the {\it unique} connected component of good vertices}\\
\text{in $\ballZZ(0,K_{N,k})\setminus\ballZZ(0,K_{N,k-1})$ such that}\\
\text{every $*$-path in $\ZZ$ from $\ballZZ(0,K_{N,k-1})$ to $\intbb\ballZZ(0,K_{N,k})$ intersects $\mathcal S_k$.}
\end{array}
\end{equation}
We will use this definition of $\mathcal S_k$ here.
If $\mathcal H_N$ does not occur, we set $\mathcal S_k = \emptyset$ for all $k$. 
Note that by Definition \ref{def:HN} the sets $\mathcal S_k$ are disjoint subsets of $\ZZ$, and for each $1\leq k\leq k_N$, 
\begin{equation}\label{def:GN}
\text{$\mathcal S_1,\ldots, \mathcal S_k$ are in the same connected component of 
good vertices in $\ballZZ(0,K_{N,k})$.}
\end{equation}

\medskip

In terms of connectivities in $\Z^d$, the key property of $\mathcal S_k$ can be stated as follows: 
if the event $\mathcal H_N$ occurs, then for each $1\leq k\leq k_N$,
\begin{equation}\label{eq:Skproperty}
\begin{array}{c}
\text{every nearest-neighbor path in $\Z^d$ from $\ballZ(0,(2R+1) K_{N,k-1})$}\\ 
\text{to $\intb\ballZ(0,(2R+1) K_{N,k})$ intersects the set $\cup_{x'\in\mathcal S_k}\ballZ(x',R)$.}
\end{array}
\end{equation}
By \eqref{eq:Sk}, \eqref{def:GN},  Definition~\ref{def:edges} and Definition~\ref{def:good},
 if $\mathcal H_N$ occurs, then 
for each $1\leq k\leq k_N$,
\begin{equation}\label{eq:Ck}
\begin{array}{c}
\text{the sets $\cup_{x_1'\in\mathcal S_1}\edges(x_1'), \ldots, \cup_{x_k'\in\mathcal S_k}\edges(x_k')$ 
are in the same connected component}\\ 
\text{of $\mathcal V^u\cap\ballZ(0,(2R+1)K_{N,k}+R)$, 
which we denote by $\mathcal C_k$.}
\end{array}
\end{equation}
If $\mathcal H_N$ does not occur, we define $\mathcal C_k = \emptyset$. 
By \eqref{eq:Ck}, 
\begin{equation}\label{eq:CkCk+1}
\text{$\mathcal C_{k}\subseteq \mathcal C_{k+1}$ for all $1\leq k\leq k_N-1$.}
\end{equation}

As we will see in Section~\ref{sec:proof:sts:int:2}, in order to prove \eqref{eq:sts:inerlacement:2}, 
it suffices to show that, with high probability, $\mathcal C_{k_N}$ is the only connected component of 
$\mathcal V^u\cap\ballZ(0,(2R+1)\cdot 2N+R)$ that 
intersects $\ballZ(0,(2R+1)\cdot N)$ and $\intb \ballZ(0,(2R+1)\cdot 2N)$. 
To prove the latter statement, we need a more general definition. 

\begin{definition}\label{def:Ak}
For $z\in \ballZ(0,(2R+1)\cdot N)$ and $1\leq k\leq k_N$, let $\mathcal A_{z,k}$ be the event that
\begin{enumerate}\itemsep1pt
\item 
$\mathcal H_N$ occurs,
\item
$z$ is connected to $\intb \ballZ(0,(2R+1)\cdot K_{N,k})$ by a nearest-neighbor path in $\mathcal V^u$, 
\item
$z\notin \mathcal C_k$, 
\end{enumerate}
and let $\mathcal A_{z,0}=\mathcal H_N$. 
\end{definition}
The main result of this section is the following lemma. 
\begin{lemma}\label{l:eventA:estimate}
For $R$ and $u$ as in \eqref{eq:Ru_1}, there exists 
$\gamma = \gamma(d,R)>0$ such that 
\begin{equation}\label{eq:strongsupercriticality}
\mathbb P\left[\mathcal A_{z,k_N} \right] \leq (1-\gamma)^{k_N}, \quad
\text{for all $N\geq 1$ and $z\in \ballZ(0,(2R+1)\cdot N)$.} 
\end{equation}
\end{lemma}
\begin{proof}[Proof of Lemma~\ref{l:eventA:estimate}]
Fix $z\in \ballZ(0,(2R+1)\cdot N)$. 
Without loss of generality we may assume that $\mathbb P\left[\mathcal A_{z,k_N}\right]\neq 0$. 
By \eqref{eq:CkCk+1}, we have the inclusion 
\begin{equation}\label{eq:AkAk+1}
\text{$\mathcal A_{z,k}\subseteq\mathcal A_{z,k-1}$,  for all $1\leq k\leq k_N$.}
\end{equation}
Using \eqref{eq:AkAk+1}, we obtain 
\begin{equation*}
\mathbb P\left[\mathcal A_{z,k_N}\right]
=
\mathbb P\left[\mathcal H_N\right]\cdot \prod_{k=1}^{k_N} \mathbb P\left[\mathcal A_{z,k}~|~\mathcal A_{z,k-1} \right] .\
\end{equation*}
To complete the proof of \eqref{eq:strongsupercriticality} it suffices to show that for all 
$z\in \ballZ(0,(2R+1)\cdot N)$, $1\leq k\leq k_N$ and some $\gamma = \gamma(d,R)>0$, 
\begin{equation}\label{eq:strongsupercriticality:1}
\mathbb P\left[\mathcal A_{z,k}~|~\mathcal A_{z,k-1} \right] \leq 1 - \gamma .\
\end{equation}
This follows from the more general Lemma~\ref{l:eventA:estimate:k} below. 
Before we state the lemma, we need some notation. 

\medskip

Define the random variable $\Sigmagood~:~\Omega\to\{0,1\}^{\ballZZ(0,2N)}$ which keeps track of good and bad vertices in 
$\ballZZ(0,2N)$ as 
\begin{equation}\label{def:sigmagood}
\Sigmagood = \left(\mathds{1}_{\{x'\mathrm{~is~good~for~}\local^u\}}~:~ x'\in\ballZZ(0,2N)\right) .\
\end{equation}
Note that 
\begin{equation}\label{eq:HNinsigmagood}
\text{$\mathcal H_N\in\sigma(\Sigmagood),\quad$ for all $N$,}  
\end{equation}
and, in particular, 
\begin{equation}\label{eq:Sksigmagood}
\text{for all $1\leq k\leq k_N$, the set $\mathcal S_k$ is measurable with respect to $\sigmagood$.}
\end{equation}
For $1\leq k\leq k_N$, if $\mathcal H_N$ occurs, 
\begin{equation}\label{def:Dk}
\begin{array}{c}
\text{let $\mathcal D_k$ be the (unique) connected component of $\Z^d \setminus \cup_{x'\in \mathcal S_k} \ballZ(x',R)$}\\
\text{ which contains the origin,}
\end{array}
\end{equation}
and let $\mathcal D_k = \ballZ(0,(2R+1)\cdot K_{N,k} -R)$ otherwise. 
\begin{figure}
\begin{center}
\psfragscanon
\includegraphics[height=8cm]{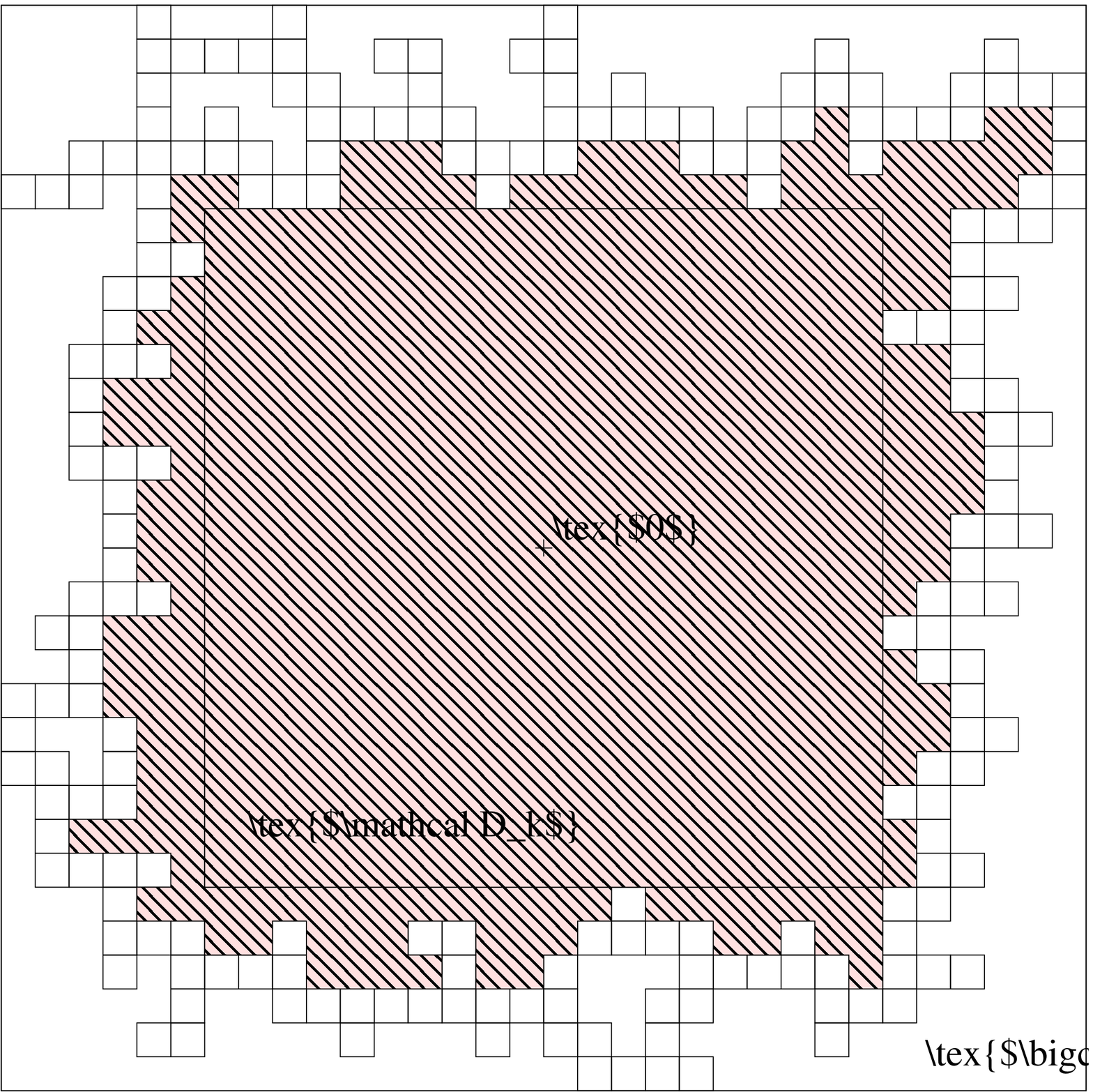}
\caption{
The inner and outer boxes are $\ballZ(0,(2R+1)\cdot K_{N,k-1} + R)$ and $\ballZ(0,(2R+1)\cdot K_{N,k} + R)$, respectively. 
The set $\mathcal D_k\subseteq \Z^d$ is the unique connected component of 
$\Z^d\setminus \cup_{x'\in\mathcal S_k}\ballZ(x',R)$, which contains the origin.}
\end{center}
\end{figure}
By 
\eqref{eq:Sksigmagood}, 
\begin{equation}\label{eq:Dksigmagood}
\text{$\mathcal D_k$ is measurable with respect to $\sigmagood$, for all $1\leq k\leq k_N$.}
\end{equation}
By \eqref{eq:Skproperty}, 
\begin{equation}\label{eq:Dk:location}
\ballZ(0,(2R+1)\cdot K_{N,k-1}+R) \subseteq 
\mathcal D_k \subseteq \ballZ(0,(2R+1)\cdot K_{N,k} -R) .\
\end{equation}
Define the random variables $\Sigmaall_k ~:~\Omega\to\{0,1\}^{\ballZ(0,(2R+1)\cdot K_{N,k} -R)}$ 
which keep track of the interlacement configuration inside $\mathcal D_k$
as 
\begin{equation}\label{def:Sigmaall}
\Sigmaall_k = \left(\mathds{1}_{\{x\in\mathcal I^u\cap\mathcal D_k\}}~:~x\in \ballZ(0,(2R+1)\cdot K_{N,k} -R) \right),
\quad 1\leq k\leq k_N .\ 
\end{equation}
The following lemma implies \eqref{eq:strongsupercriticality:1}, 
as we show in \eqref{eq:l:eventA:estimate:k:implies:eq:strongsupercriticality:1}. 
\begin{lemma}\label{l:eventA:estimate:k}
There exists $\gamma=\gamma(d,R)>0$ such that for all $z\in \ballZ(0,(2R+1)\cdot N)$ and $1\leq k\leq k_N$, 
\begin{equation}\label{eq:strongsupercriticality:k}
\mathds{1}_{\mathcal H_N} \cdot \mathbb P\left[\mathcal A_{z,k}~|~\Sigmagood,\Sigmaall_k \right] \leq 1-\gamma, \quad\quad
\mathbb P\mbox{-a.s.}
\end{equation}
\end{lemma}
We postpone the proof of Lemma~\ref{l:eventA:estimate:k} until Section~\ref{sec:eventA}, and now complete the proof 
of Lemma~\ref{l:eventA:estimate} by showing how Lemma~\ref{l:eventA:estimate:k} implies \eqref{eq:strongsupercriticality:1}.

By \eqref{eq:Ck}, Definition~\ref{def:Ak}, \eqref{eq:HNinsigmagood}, \eqref{eq:Dksigmagood}, and \eqref{eq:Dk:location}, we have
\begin{equation}\label{eq:Aksigma}
\text{$\mathcal A_{z,k-1}\in \sigma(\Sigmagood,\Sigmaall_{k})$, for each $1\leq k\leq k_N$.}
\end{equation}
Therefore, for each $1\leq k\leq k_N$, 
\begin{equation}\label{eq:l:eventA:estimate:k:implies:eq:strongsupercriticality:1}
\mathbb P\left[\mathcal A_{z,k} \right] 
\stackrel{\eqref{eq:AkAk+1}, \eqref{eq:Aksigma}}= 
\mathbb E\left[\mathds{1}_{\mathcal A_{z,k-1}}\cdot 
\mathbb P\left[\mathcal A_{z,k}~|~ \Sigmagood,\Sigmaall_{k}\right]\right]
\stackrel{\eqref{eq:strongsupercriticality:k}}\leq (1-\gamma)\cdot \mathbb P\left[\mathcal A_{z,k-1}\right] .\
\end{equation}
This implies \eqref{eq:strongsupercriticality:1} and completes the proof of Lemma~\ref{l:eventA:estimate}
subject to Lemma~\ref{l:eventA:estimate:k}, which will be proved in Section~\ref{sec:eventA}. 
\end{proof}

\bigskip

\subsection{Proof of Lemma~\ref{l:eventA:estimate:k}}\label{sec:eventA}

In this section we prove Lemma~\ref{l:eventA:estimate:k}. 
Recall the definitions of the configuration $\Sigmagood$ of good and bad vertices of $\ballZZ(0,2N)$
(see \eqref{def:sigmagood}), the event $\mathcal H_N$ (see Definition~\ref{def:HN})
 guaranteeing the presence of 
$k_N=\lfloor \sqrt{N} \rfloor$ connected shells $\mathcal S_k$, $1 \leq k \leq k_N$  of good boxes 
(see \eqref{eq:Sk}), the domain $\mathcal D_k\subseteq \Z^d$ surrounded by  
$\cup_{x' \in \mathcal S_k} \ballZ(x',R)$ (see \eqref{def:Dk}),
 and the configuration $\Sigmaall_k$ of occupied/vacant vertices of $\mathcal D_k$ (see \eqref{def:Sigmaall}).

\medskip

The occurrence of event $\mathcal A_{z,k}$ guarantees the existence of a vacant path in $\mathcal D_k$ 
from $z$ to $\intb\mathcal D_k$ with certain restrictions on the location of the end point of this path on $\intb \mathcal D_k$. 
These properties are reflected in the following definition.

\begin{definition}\label{def:tildeAk}
For $z\in \ballZ(0,(2R+1)\cdot N)$, and $1\leq k\leq k_N$, 
let $\widetilde {\mathcal A}_{z,k}$ be the event that 
(a) $\mathcal H_N$ occurs, and 
(b) there exists a nearest-neighbor path $\pi_k$ in $\mathcal D_k$ 
from $z$ to a vertex $x_k\in\intb\mathcal D_k\setminus \extb\cup_{x'\in\mathcal S_k} \edges(x')$ 
such that every vertex $x$ along this path (including $x_k$)
 satisfies $\Sigmaall_k(x) = 0$ (i.e. $x\in \mathcal{V}^u$, c.f. \eqref{def:Sigmaall}).
If there are several such paths, we pick one in a predetermined, non-random fashion. 
\end{definition}
The properties of $\widetilde {\mathcal A}_{z,k}$ that are useful to us are the following: 
\begin{equation}\label{eq:tildeAk:meas}
\text{$\widetilde {\mathcal A}_{z,k}$ (and hence $\pi_k$ and $x_k$) 
is measurable with respect to $\sigma(\Sigmagood,\Sigmaall_k)$} ,\ 
\end{equation}
and, by Definition~\ref{def:Ak},
\begin{equation}\label{eq:AktildeAk}
\mathcal A_{z,k} \subseteq \widetilde {\mathcal A}_{z,k} .\
\end{equation}
Indeed, \eqref{eq:tildeAk:meas} is immediate from Definition~\ref{def:tildeAk}. 
To see that \eqref{eq:AktildeAk} holds, note that 
if $\mathcal A_{z,k}$ occurs, then by \eqref{eq:Dk:location} $z$ is connected to $\intb\mathcal D_k$ 
by a nearest-neighbor path of vertices $x$ with $\Sigmaall_k(x) = 0$. 
However, by \eqref{eq:Ck}, $\cup_{x'\in\mathcal S_k} \edges(x')\subseteq \mathcal C_k$ and, by Definition~\ref{def:Ak}, 
$z\notin\mathcal C_k$,  therefore any such path must avoid $\extb\cup_{x'\in\mathcal S_k} \edges(x')$. 
This implies \eqref{eq:AktildeAk}.

\medskip

By the definition of
 $\widetilde {\mathcal A}_{z,k}$, $x_k\in\intb\mathcal D_k\setminus\extb\cup_{x'\in\mathcal S_k} \edges(x')$. 
Therefore, there exists a unique 
\begin{equation}\label{def:xkQk}
\begin{array}{c}
\text{$x_k'\in\mathcal S_k$ such that 
$x_k$ belongs to the exterior boundary of $\cube_k = \cube(x_k')$ (see \eqref{def:cube})}\\
\text{and is not adjacent to any of the vertices in $\edges(x_k')$.}
\end{array}
\end{equation}
Also there exists a (unique) $\widetilde x_k\in \cube(x_k')\setminus\edges(x_k')$ such that 
$x_k\sim\widetilde x_k$. 
\begin{figure}
\begin{center}
\psfragscanon
\includegraphics[height=8cm]{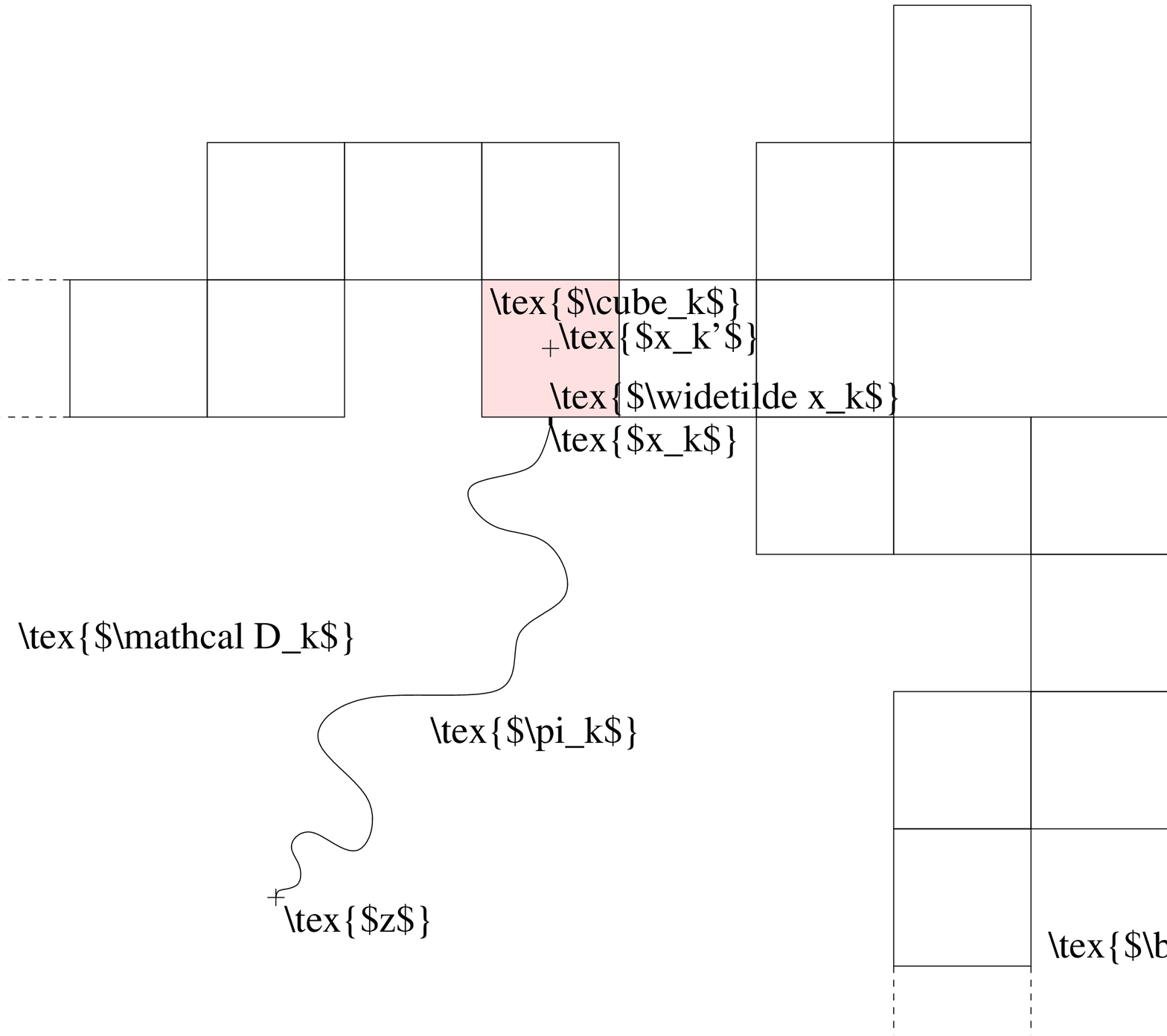}
\caption{If the event $\widetilde {\mathcal A}_{z,k}$ occurs, there exists a vacant path $\pi_k$ from 
$z$ to $x_k\in \intb\mathcal D_k$ in $\mathcal D_k$ such that 
$x_k\notin\extb\cup_{x'\in\mathcal S_k}\edges(x')$. 
There exists a unique $x_k'\in\mathcal S_k$ such that $x_k\in\extb\cube(x_k')$ (and $x_k\notin\extb\edges(x_k')$). 
The cube $\cube(x_k')$ is denoted by $\cube_k$. 
The unique neighbor of $x_k$ in $\intb \cube_k$ is denoted by $\widetilde x_k$.}
\end{center}
\end{figure}
Moreover, since $\widetilde x_k\notin \edges(x_k')$, 
\begin{equation}\label{eq:xkwidetildexk:uniqueness}
\text{$x_k$ is the only nearest-neighbor of $\widetilde x_k$ which is outside $\cube(x_k')$.}
\end{equation}

\bigskip

The key step in the proof of Lemma~\ref{l:eventA:estimate:k} is Lemma~\ref{l:Pkbad}, 
in which we show that given the configurations $\Sigmagood$ of good and bad vertices of $\ballZZ(0,2N)$ 
and $\Sigmaall_k$ of occupied/vacant vertices of $\mathcal D_k$ satisfying the event 
$\widetilde {\mathcal A}_{z,k}$, and given 
the $\sigma$-algebra generated by the interlacement excursions outside $\cube(x_k')$, 
with uniformly positive probability there is a realization of the interlacement excursions inside $\cube(x_k')$ 
such that $x_k'$ is good, and $x_k$ is connected to $\edges(x_k')$ in $\mathcal V^u\cap(\cube(x_k')\cup\{x_k\})$. 
Once this is done, Lemma~\ref{l:eventA:estimate:k} immediately follows, as we show after the statement 
of Lemma~\ref{l:Pkbad}. To state Lemma~\ref{l:Pkbad}, we need some notation.

\bigskip

\begin{definition}\label{def:XinnKXABK}
Let $K\subset\subset\Z^d$. In the notation of Section~\ref{sec:condindep:2}, let $X^\inn_K$ be the (random) 
vector 
\[
X^\inn_K = \left(X^\inn_{i,j}~:~1\leq i\leq N_{K,u},~ 1\leq j\leq M_i\right)
\]
of the excursions inside $K$ of the interlacement trajectories from the support of $\omega_{K,u}$ 
(numbered in order of increase of their labels), and 
$X^{AB}_K$ the vector
\[
X^{AB}_K = \left((X_i(A_{i,j}),X_i(B_{i,j}))~:~1\leq i\leq N_{K,u},~ 1\leq j\leq M_i \right)
\]
of start and end points of all these excursions. 
Note that 
\begin{equation}\label{eq:XinnKXABK:meas}
\text{$X^\inn_K$ is measurable with respect to $\mathcal F^\inn_{K,u}$, and 
$X^{AB}_K$ with respect to $\mathcal F^{AB}_{K,u}$,}
\end{equation} 
with $\mathcal F^\inn_{K,u}$ and $\mathcal F^{AB}_{K,u}$ defined in Section~\ref{sec:condindep:3}.

\end{definition}
\begin{definition}\label{def:Tinnx'}
For $x'\in\ZZ$, let 
$\mathcal T^\inn_{x'} = \mathcal T^\inn_{x'}(X^{AB}_{\cube(x')})$ be the set of all vectors 
\[
\tau^\inn = \left(\tau^\inn_{i,j}~:~1\leq i\leq N_{\cube(x'),u},~ 1\leq j\leq M_i\right)
\]
of finite nearest-neighbor trajectories from $X_i(A_{i,j})$ to $X_i(B_{i,j})$ inside $\cube(x')$ 
such that
\begin{itemize}
\item[(a)]
all the $\tau^\inn_{i,j}$ avoid $\edges(x')$, and 
\item[(b)]
the total number of visits to $\intb\cube(x')$ of all the $\tau^\inn_{i,j}$ is at most $R^{d-1}$.
\end{itemize}
Note that 
\begin{equation}\label{eq:Tinnx':meas}
\text{$\mathcal T^\inn_{x'}$ is measurable with respect to $\mathcal F^{AB}_{\cube(x'),u}$} 
\end{equation}
(see Section~\ref{sec:condindep:3}), 
and by Definition~\ref{def:good}, for any $x'\in \ZZ$, 
\begin{equation}\label{eq:Tinn:good}
\{X^\inn_{\cube(x')} \in \mathcal T^\inn_{x'}\}= \{\text{$x'$ is good for $\local^u$}\} .\
\end{equation}
\end{definition}
\begin{claim}
Recall the definition of $x_k'$ and $\cube_k$ from \eqref{def:xkQk}. 

(1) If $\widetilde {\mathcal A}_{z,k}$ occurs, then for all $1\leq i\leq N_{\cube_k,u}$, $1\leq j\leq M_i$, and 
for any element $(X_i(A_{i,j}),X_i(B_{i,j}))$ of $X^{AB}_{\cube_k}$, we have
\begin{equation}\label{eq:Ek:property}
X_i(A_{i,j}), X_i(B_{i,j})\in \intb\cube_k\setminus\left(\edges(x_k')\cup\{\widetilde x_k\}\right) .\
\end{equation}
Indeed, if $\widetilde {\mathcal A}_{z,k}$ occurs, then $x_k'$ is good for $\local^u$ and, 
by Definition~\ref{def:tildeAk}, $x_k\in\mathcal V^u$. 
Together with \eqref{eq:xkwidetildexk:uniqueness}, this implies \eqref{eq:Ek:property}. 

(2) If $\widetilde {\mathcal A}_{z,k}$ occurs, then 
\begin{equation}\label{eq:PkTk}
X^\inn_{\cube_k} \in \mathcal T^\inn_{x_k'}. 
\end{equation}
Indeed, \eqref{eq:PkTk} follows from \eqref{eq:Tinn:good} and the fact that 
the vertex $x_k'$ is good for $\local^u$ when $\widetilde {\mathcal A}_{z,k}$ occurs. 
\end{claim}

\bigskip

Lemma~\ref{l:eventA:estimate:k} follows from the next lemma. 
Recall Definition~\ref{def:Ak} of the event $\mathcal A_{z,k}$, 
the definition of $x_k'$ and $\cube_k$ from \eqref{def:xkQk}, and 
the notion of the $\sigma$-algebra $\mathcal F^\outt_{K,u}$ generated by the interlacement excursions outside of 
$K\subset\subset \Z^d$ and $\omega-\omega_{K,u}$ from Section~\ref{sec:condindep:3}. 

\begin{lemma}\label{l:Pkbad}
There exists $\gamma = \gamma(d,R)>0$ such that
for any $z\in\ballZ(0,(2R+1)\cdot N)$ and $1\leq k\leq k_N$, 
$\mathbb P$-almost surely,
for each realization of $\Sigmagood$, $\Sigmaall_k$, and $X^{AB}_{\cube_k}$ 
satisfying $\widetilde {\mathcal A}_{z,k}$, there exists  
\begin{equation}\label{eq:Pkbad:0}
\rho^\inn = \rho^\inn(\Sigmagood,\Sigmaall_k,X^{AB}_{\cube_k})\in\mathcal T^\inn_{x_k'}
\end{equation}
such that for all $x'\in\ZZ$,
\begin{equation}\label{eq:Pkbad:1}
\mathds{1}_{\widetilde {\mathcal A}_{z,k}\cap\{x_k' = x'\}}\cdot \mathbb P\left[
X^\inn_{\cube_k} = \rho^\inn
~|~\sigma(\Sigmagood, \Sigmaall_k, \mathcal F^\outt_{\cube(x'),u})\right]
\geq 
\mathds{1}_{\widetilde {\mathcal A}_{z,k}\cap\{x_k' = x'\}}\cdot\gamma 
\end{equation}
and 
\begin{equation}\label{eq:Pkbad:2}
\widetilde {\mathcal A}_{z,k}\cap\{X^\inn_{\cube_k} = \rho^\inn\} \subseteq 
\widetilde {\mathcal A}_{z,k}\setminus \mathcal A_{z,k} .\
\end{equation}
\end{lemma}
Before we prove Lemma~\ref{l:Pkbad}, we use it to finish the proof of Lemma~\ref{l:eventA:estimate:k}. 
We have 
\begin{eqnarray*}\label{eq:Pkrhok}
\mathds{1}_{\mathcal H_N} \cdot \mathbb P\left[\mathcal A_{z,k}~|~\Sigmagood, \Sigmaall_k \right]
&\stackrel{\eqref{eq:tildeAk:meas},\eqref{eq:AktildeAk}}=
&\mathds{1}_{\widetilde {\mathcal A}_{z,k}} \cdot \mathbb P\left[\mathcal A_{z,k}~|~\Sigmagood, \Sigmaall_k \right]\\
&\stackrel{\eqref{eq:Pkbad:2}}\leq
&\mathds{1}_{\widetilde {\mathcal A}_{z,k}} \cdot 
\mathbb P\left[X^\inn_{\cube_k}\neq \rho^\inn~|~\Sigmagood, \Sigmaall_k \right] \\
&\stackrel{\eqref{eq:Pkbad:1}}\leq 
&\mathds{1}_{\widetilde {\mathcal A}_{z,k}} \cdot (1 - \gamma) .\
\end{eqnarray*}
This finishes the proof of Lemma~\ref{l:eventA:estimate:k}, subject to Lemma~\ref{l:Pkbad}. 

\bigskip

It remains to prove Lemma~\ref{l:Pkbad}. We begin with some preliminary results. 
Recall the notion of the $\sigma$-algebras 
$\mathcal F^\inn_{K,u}$, $\mathcal F^\outt_{K,u}$, and $\mathcal F^{AB}_{K,u}$ from
Section~\ref{sec:condindep:3}.
\begin{lemma}\label{l:condindep:good}
For any $x'\in\ZZ$ and $\mathcal E^\inn\in \mathcal F^\inn_{\cube(x'),u}$, we have, $\mathbb P$-almost surely, that
\begin{multline}\label{eq:condindep:good}
\mathds{1}_{\widetilde {\mathcal A}_{z,k}\cap\{x_k' = x'\}}\cdot
\mathbb P\left[\mathcal E^\inn~|~\sigma(\Sigmagood, \Sigmaall_k, \mathcal F^\outt_{\cube(x'),u})\right]\\
=
\mathds{1}_{\widetilde {\mathcal A}_{z,k}\cap\{x_k' = x'\}}\cdot
\frac
{\mathbb P\left[\mathcal E^\inn\cap\{\text{$x'$ is good for $\local^u$}\}~|~\mathcal F^{AB}_{\cube(x'),u}\right]}
{\mathbb P\left[\{\text{$x'$ is good for $\local^u$}\}~|~\mathcal F^{AB}_{\cube(x'),u}\right]} .\
\end{multline}
\end{lemma}
\begin{remark}
Note that $\widetilde {\mathcal A}_{z,k}\cap\{x_k' = x'\} \subseteq \{\text{$x'$ is good for $\local^u$}\}$. Therefore, 
\[
\mathbb P\left[
\widetilde {\mathcal A}_{z,k}\cap\{x_k' = x'\} \cap
\left\{\omega~:~\mathbb P\left[\{\text{$x'$ is good for $\local^u$}\}~|~\mathcal F^{AB}_{\cube(x'),u}\right] = 0\right\} 
\right] = 0 .\
\]
\end{remark}
\begin{proof}[Proof of Lemma~\ref{l:condindep:good}]
Let $z\in \ballZ(0,(2R+1)\cdot N)$. 
Let 
\begin{multline*}
\sigmagoodel\in\{0,1\}^{\ballZZ(0,2N)}, \quad \sigmaallel_k\in\{0,1\}^{\ballZ(0,(2R+1)\cdot K_{N,k} -R)},\\
\partial_k\subseteq \ballZ(0,(2R+1)\cdot K_{N,k} -R), \quad \mbox{and}\quad x'\in\ballZZ(0,2N)
\end{multline*}
be such that 
\begin{equation}\label{eq:Sigmas:choice}
\{\Sigmagood = \sigmagoodel,~ \Sigmaall_k = \sigmaallel_k\}\subseteq 
\widetilde {\mathcal A}_{z,k}\cap\{x_k' = x'\}\cap\{\mathcal D_k = \partial_k\} ,\
\end{equation}
where $\mathcal D_k$ is defined in \eqref{def:Dk}. 
Let 
\[
K = \cube(x') .\
\]
In order to prove \eqref{eq:condindep:good}, 
it suffices to show that for any events $\mathcal E^\inn\in\mathcal F^\inn_{K,u}$ and $\mathcal E^\outt\in\mathcal F^\outt_{K,u}$, 
we have 
\begin{multline}\label{eq:condindep:good:1}
\mathbb P\left[
\mathcal E^\inn\cap\mathcal E^\outt\cap\{\Sigmagood = \sigmagoodel,~ \Sigmaall_k = \sigmaallel_k\}
\right]\\
=
\mathbb E \left[
\frac
{\mathbb P\left[\mathcal E^\inn\cap\{\text{$x'$ is good for $\local^u$}\}~|~\mathcal F^{AB}_{K,u}\right]}
{\mathbb P\left[\{\text{$x'$ is good for $\local^u$}\}~|~\mathcal F^{AB}_{K,u}\right]} ;~
\mathcal E^\outt\cap\{ \Sigmagood = \sigmagoodel,~ \Sigmaall_k = \sigmaallel_k\}
\right] .\
\end{multline}
Let 
\[
\Sigmagoodout = 
\left(\mathds{1}_{\{\widetilde x'\mathrm{~is~good~for~}\mathcal L^u\}}~:~
\widetilde x'\in \ballZZ(0,2N)\setminus\{x'\}\right)
\]
be the restriction of $\Sigmagood$ to $\ballZZ(0,2N)\setminus\{x'\}$, and 
let $\sigmagoodelout\in \{0,1\}^{\ballZZ(0,2N)\setminus\{x'\}}$ be 
the restriction of $\sigmagoodel$ to $\ballZZ(0,2N)\setminus\{x'\}$. 
Consider the events 
\begin{align*}
\widetilde {\mathcal E}^\inn &= \{x'\mathrm{~is~good~for~}\mathcal L^u\} ,\\
\widetilde {\mathcal E}^\outt &= 
\left\{\Sigmagoodout = \sigmagoodelout,~ 
\left(\mathds{1}_{\{x\in\mathcal I^u\cap\partial_k\}}~:~x\in \ballZ(0,(2R+1)\cdot K_{N,k} -R)\right) = \sigmaallel_k
\right\} .\
\end{align*}
Note that by Claim~\ref{claim:sigmaF} (3) and Definition~\ref{def:good}, we have 
\begin{equation}\label{eq:Einn:meas}
\widetilde {\mathcal E}^\inn \in \mathcal F^\inn_{K,u} ,\
\end{equation}
by Claim~\ref{claim:sigmaF} (3) and the fact that $\partial_k\cap K = \emptyset$, 
\begin{equation}\label{eq:Eoutt:meas}
\widetilde {\mathcal E}^\outt \in\mathcal F^\outt_{K,u} ,\
\end{equation}
and by \eqref{eq:Sigmas:choice}, 
\begin{equation}\label{eq:Sigmas:EinnEoutt}
\{\Sigmagood = \sigmagoodel,~ \Sigmaall_k = \sigmaallel_k\} 
= \widetilde {\mathcal E}^\inn\cap\widetilde {\mathcal E}^\outt .\
\end{equation}
Using these observations and Lemma~\ref{l:condindep} (a), we rewrite the left-hand side of \eqref{eq:condindep:good:1} as 
\begin{multline*}
\mathbb P\left[
\mathcal E^\inn\cap\mathcal E^\outt\cap\{\Sigmagood = \sigmagoodel,~ \Sigmaall_k = \sigmaallel_k\}
\right]
\stackrel{\eqref{eq:Sigmas:EinnEoutt}}=
\mathbb P\left[
(\mathcal E^\inn\cap\widetilde {\mathcal E}^\inn)\cap
(\mathcal E^\outt\cap\widetilde {\mathcal E}^\outt)
\right]\\
\stackrel{\mathrm{Lemma}~\ref{l:condindep}\mathrm{(a)},\eqref{eq:Einn:meas}, \eqref{eq:Eoutt:meas}}=
\mathbb E \left[
\mathbb P\left[
\mathcal E^\inn\cap\widetilde {\mathcal E}^\inn~|~\mathcal F^{AB}_{K,u}
\right]
\cdot
\mathbb P\left[
\mathcal E^\outt\cap\widetilde {\mathcal E}^\outt~|~\mathcal F^{AB}_{K,u}
\right]
\right]\\
=
\mathbb E \left[
\frac{\mathbb P\left[
\mathcal E^\inn\cap\widetilde {\mathcal E}^\inn~|~\mathcal F^{AB}_{K,u}
\right]}
{\mathbb P\left[
\widetilde {\mathcal E}^\inn~|~\mathcal F^{AB}_{K,u}
\right]}
\cdot
\mathbb P\left[
\widetilde {\mathcal E}^\inn~|~\mathcal F^{AB}_{K,u}
\right]
\cdot
\mathbb P\left[
\mathcal E^\outt\cap\widetilde {\mathcal E}^\outt~|~\mathcal F^{AB}_{K,u}
\right]
\right]\\
\stackrel{\mathrm{Lemma}~\ref{l:condindep}\mathrm{(a)},\eqref{eq:Einn:meas}, \eqref{eq:Eoutt:meas}}=
\mathbb E \left[
\frac{\mathbb P\left[
\mathcal E^\inn\cap\widetilde {\mathcal E}^\inn~|~\mathcal F^{AB}_{K,u}
\right]}
{\mathbb P\left[
\widetilde {\mathcal E}^\inn~|~\mathcal F^{AB}_{K,u}
\right]}
\cdot
\mathbb P\left[
\widetilde {\mathcal E}^\inn\cap \mathcal E^\outt\cap\widetilde {\mathcal E}^\outt~|~\mathcal F^{AB}_{K,u}
\right]
\right]\\
\stackrel{\eqref{eq:Sigmas:EinnEoutt}}=
\mathbb E \left[
\frac{\mathbb P\left[
\mathcal E^\inn\cap\widetilde {\mathcal E}^\inn~|~\mathcal F^{AB}_{K,u}
\right]}
{\mathbb P\left[
\widetilde {\mathcal E}^\inn~|~\mathcal F^{AB}_{K,u}
\right]}
\cdot
\mathbb P\left[
\mathcal E^\outt\cap\{\Sigmagood = \sigmagoodel,~ \Sigmaall_k = \sigmaallel_k\}~|~\mathcal F^{AB}_{K,u}
\right]
\right]\\
=
\mathbb E\left[
\frac{\mathbb P\left[
\mathcal E^\inn\cap\widetilde {\mathcal E}^\inn~|~\mathcal F^{AB}_{K,u}
\right]}
{\mathbb P\left[
\widetilde {\mathcal E}^\inn~|~\mathcal F^{AB}_{K,u}
\right]};~
\mathcal E^\outt\cap\{\Sigmagood = \sigmagoodel,~ \Sigmaall_k = \sigmaallel_k\}
\right] .\
\end{multline*}
This is precisely \eqref{eq:condindep:good:1}. 
The proof of Lemma~\ref{l:condindep:good} is complete. 
\end{proof}

\medskip

\begin{lemma}\label{l:Pk}
For $z\in\ballZ(0,(2R+1)\cdot N)$, $x'\in\ZZ$, non-negative integers $n$ and $(m_i~:~1\leq i\leq n)$, 
vector $\tau^\inn = (\tau^\inn_{i,j}~:~1\leq i\leq n,~1\leq j\leq m_i)$ of 
finite nearest-neighbor trajectories $\tau^\inn_{i,j}$ in $\cube(x')$ 
from $x_{i,j}\in\intb \cube(x')$ to $y_{i,j}\in\intb \cube(x')$,  
and $1\leq k\leq k_N$, we have
$\mathbb P$-almost surely, that 
\begin{multline}\label{eq:Pk:lowerbound}
\mathds{1}_{\widetilde {\mathcal A}_{z,k}\cap\{x_k' = x'\}}\cdot
\mathbb P\left[X^\inn_{\cube_k} = \tau^\inn~|~\sigma(\Sigmagood, \Sigmaall_k, \mathcal F^\outt_{\cube(x'),u})\right]\\
\geq 
\mathds{1}_{\widetilde {\mathcal A}_{z,k}\cap\{x_k' = x'\}\cap\{\tau^\inn \in \mathcal T^\inn_{x'}\}}\cdot
\prod_{i=1}^{n}\prod_{j=1}^{m_i}
\left(1/2d\right)^{|\tau^\inn_{i,j}|} .\
\end{multline}
\end{lemma}
\begin{remark}
Note that by \eqref{eq:PkTk}, we have 
\[
\widetilde {\mathcal A}_{z,k}\cap\{x_k' = x'\}\cap \{X^\inn_{\cube_k} = \tau^\inn\}\subseteq\{\tau^\inn \in \mathcal T^\inn_{x'}\} ,\ 
\]
and by \eqref{eq:Tinnx':meas} and Claim~\ref{claim:sigmaF}(1), 
\[
\{\tau^\inn \in \mathcal T^\inn_{x'}\}\in \mathcal F^{AB}_{\cube(x'),u}\subset \mathcal F^\outt_{\cube(x'),u} .\  
\]
In particular, the right-hand side of \eqref{eq:Pk:lowerbound} is measurable with respect to 
$\sigma(\Sigmagood,\Sigmaall_k,\mathcal F^{AB}_{\cube(x'),u})$. 
\end{remark}

\begin{proof}[Proof of Lemma~\ref{l:Pk}]
By \eqref{eq:XinnKXABK:meas}, 
$\{X^\inn_{\cube(x')} = \tau^\inn\}\in\mathcal F^\inn_{\cube(x'),u}$. Using Lemma~\ref{l:condindep:good}, 
we obtain
\begin{multline}\label{eq:Xinncubexk}
\mathds{1}_{\widetilde {\mathcal A}_{z,k}\cap\{x_k' = x'\}}\cdot
\mathbb P\left[X^\inn_{\cube_k} = \tau^\inn~|~\sigma(\Sigmagood, \Sigmaall_k, \mathcal F^\outt_{\cube(x'),u})\right]\\
=
\mathds{1}_{\widetilde {\mathcal A}_{z,k}\cap\{x_k' = x'\}}\cdot
\mathbb P\left[X^\inn_{\cube(x')} = \tau^\inn~|~\sigma(\Sigmagood, \Sigmaall_k, \mathcal F^\outt_{\cube(x'),u})\right]\\
\stackrel{\eqref{eq:condindep:good}}=
\mathds{1}_{\widetilde {\mathcal A}_{z,k}\cap\{x_k' = x'\}}\cdot
\frac
{\mathbb P\left[X^\inn_{\cube(x')} = \tau^\inn,~ \text{$x'$ is good for $\local^u$} ~|~\mathcal F^{AB}_{\cube(x'),u}\right]}
{\mathbb P\left[\text{$x'$ is good for $\local^u$}~|~\mathcal F^{AB}_{\cube(x'),u}\right]} \\
\stackrel{\eqref{eq:Tinn:good}}\geq
\mathds{1}_{\widetilde {\mathcal A}_{z,k}\cap\{x_k' = x'\}\cap\{\tau^\inn \in \mathcal T^\inn_{x'}\}}\cdot
\mathbb P\left[X^\inn_{\cube(x')} = \tau^\inn~|~\mathcal F^{AB}_{\cube(x'),u}\right] .\
\end{multline}
Using \eqref{eq:EinnK:proba}, we get 
\begin{multline*}
\mathds{1}_{\{\tau^\inn \in \mathcal T^\inn_{x'}\}}\cdot
\mathbb P\left[X^\inn_{\cube(x')} = \tau^\inn~|~\mathcal F^{AB}_{\cube(x'),u}\right] \\
=
\mathds{1}_{\{\tau^\inn \in \mathcal T^\inn_{x'}\}}\cdot
\prod_{i=1}^{n}\prod_{j=1}^{m_i}
\frac{P_{x_{i,j}}[(X(t)~:~0\leq t\leq T_{\cube(x')}-1)= \tau^\inn_{i,j}]}
{P_{x_{i,j}}[X(T_{\cube(x')}-1)= y_{i,j}]} \\
\geq
\mathds{1}_{\{\tau^\inn \in \mathcal T^\inn_{x'}\}}\cdot
\prod_{i=1}^{n}\prod_{j=1}^{m_i}
P_{x_{i,j}}[(X(t)~:~0\leq t\leq T_{\cube(x')}-1)= \tau^\inn_{i,j}] \\
=
\mathds{1}_{\{\tau^\inn \in \mathcal T^\inn_{x'}\}}\cdot
\prod_{i=1}^{n}\prod_{j=1}^{m_i}
P_{x_{i,j}}[(X(t)~:~0\leq t\leq |\tau^\inn_{i,j}|-1)= \tau^\inn_{i,j},~X(|\tau^\inn_{i,j}|)\notin \cube(x')]\\
\geq 
\mathds{1}_{\{\tau^\inn \in \mathcal T^\inn_{x'}\}}\cdot
\prod_{i=1}^{n}\prod_{j=1}^{m_i}
\left(1/2d\right)^{|\tau^\inn_{i,j}|} .\
\end{multline*}
Together with \eqref{eq:Xinncubexk}, this implies \eqref{eq:Pk:lowerbound} and finishes the proof of Lemma~\ref{l:Pk}. 
\end{proof}

\bigskip

\begin{proof}[Proof of Lemma~\ref{l:Pkbad}]
Fix $z\in \ballZ(0,(2R+1)\cdot N)$, $1\leq k\leq k_N$, and 
a realization of $\Sigmagood$, $\Sigmaall_k$, and $X^{AB}_{\cube_k}$ satisfying $\widetilde {\mathcal A}_{z,k}$.
Our aim is to construct $\rho^\inn = \rho^\inn(\Sigmagood, \Sigmaall_k, X^{AB}_{\cube_k})$ satisfying 
\eqref{eq:Pkbad:0}, \eqref{eq:Pkbad:1}, and \eqref{eq:Pkbad:2}. 

\medskip

We begin by defining a ``tunnel'' from $\widetilde x_k$ to $\edges(x_k')$ inside $\cube_k$,
 which we will later force to be vacant. 
Recall that $\widetilde x_k\in \intb \cube_k\setminus\edges(x_k')$. 
By Definition~\ref{def:edges}, precisely one of the coordinates of the vector $\widetilde x_k - x_k'$ is $-R$ or $R$, and 
the values of all the remaining coordinates are between $-R+3$ and $R-3$. Let $i$ be this unique coordinate, 
and let $j$ be the first among the remaining $(d-1)$ coordinates which is not $i$. 
For $1\leq s\leq d$, let $e_s$ be the $s$th unit vector. 
We define the subset $T_k$ of $\cube_k$  
to be 
\[
\{\widetilde x_k, \widetilde x_k + e_i, \widetilde x_k + 2e_i\}\cup(\{\widetilde x_k + 2e_i + te_j~:~t\geq 0\}\cap\cube_k)
\]
if the value of the $i$th coordinate of $\widetilde x_k-x_k'$ is $-R$, 
or 
\[
\{\widetilde x_k, \widetilde x_k - e_i, \widetilde x_k - 2e_i\}\cup(\{\widetilde x_k - 2e_i + te_j~:~t\geq 0\}\cap\cube_k)
\]
if the value of the $i$th coordinate of $\widetilde x_k-x_k'$ is $R$. 
Note that for $R\geq 4$, 
\begin{enumerate}
\item[(1)]
$T_k\cap\edges(x_k') \neq \emptyset$,
\item[(2)]
$\cube_k\setminus (\dcube_k\cup\edges(x_k')\cup T_k)$ is a connected subset of $\cube_k$, and 
\item[(3)]
every $x\in\dcube_k\setminus(\edges(x_k')\cup \{\widetilde x_k\})$ has a neighbor in 
$\cube_k\setminus (\dcube_k\cup\edges(x_k')\cup T_k)$. 
\end{enumerate}
In particular, (2) and (3) imply that any two points $a,b\in \dcube_k\setminus(\edges(x_k')\cup \{\widetilde x_k\})$ 
are connected by a self-avoiding path in $\{a,b\}\cup(\cube_k\setminus (\dcube_k\cup\edges(x_k')\cup T_k))$.

\bigskip

Taking into account \eqref{eq:Ek:property}, the above mentioned properties of $T_k$ imply that 
for each element $(X_i(A_{i,j}),X_i(B_{i,j}))$ of $X^{AB}_{\cube_k}$, 
there exist self-avoiding paths 
$\rho^\inn_{i,j}$ which connect $X_i(A_{i,j})$ to $X_i(B_{i,j})$ and are entirely contained in 
$\cube_k\setminus (\dcube_k\cup\edges(x_k')\cup T_k)$ except for their start and end points, 
$X_i(A_{i,j})$ and $X_i(B_{i,j})$, which are in $\dcube_k\setminus(\edges(x_k')\cup\{\widetilde x_k\})$.
(Note that if $X_i(A_{i,j})=X_i(B_{i,j})$, then $\rho^\inn_{i,j} = \{X_i(A_{i,j})\}$ is the 
unique self-avoiding path from $X_i(A_{i,j})$ to $X_i(B_{i,j})$.)
We choose one of such collections of self-avoiding paths 
$\rho^\inn = \rho^\inn(\Sigmagood,\Sigmaall_k,X^{AB}_{\cube_k})$ in a predetermined, non-random way. 

\bigskip

We will now show that $\rho^\inn$ satisfies the requirements of Lemma~\ref{l:Pkbad}. 
First we show \eqref{eq:Pkbad:0}. 
Recall Definition~\ref{def:Tinnx'} of $\mathcal T^\inn_{x'}$. 
By construction, the total number of visits of all the $\rho^\inn_{i,j}$ to $\intb\cube_k$ is the smallest one among all 
the possible collections of paths 
$\tau^\inn = (\tau^\inn_{i,j}~:~1\leq i\leq N_{\cube_k,u},~1\leq j\leq M_i)$ 
inside $\cube_k$ from $X_i(A_{i,j})$ to $X_i(B_{i,j})$. 
In particular, it is almost surely smaller or equal to the total number of visits to $\intb\cube_k$ 
by the trajectories in $X^\inn_{\cube_k}$, 
which is at most $R^{d-1}$ by \eqref{eq:PkTk}. 
Thus, $\rho^\inn$ satisfies \eqref{eq:Pkbad:0}. 

\bigskip

Now we show that $\rho^\inn$ satisfies \eqref{eq:Pkbad:2}. 
If $X^\inn_{\cube_k} = \rho^\inn$, then $T_k \subset \mathcal V^u$. 
In particular, since $\widetilde x_k$ is connected to $\edges(x_k')$ by $T_k$, 
we obtain that $\widetilde x_k$ is connected to $\edges(x_k')$ in $\mathcal V^u\cap\cube_k$. 
Recall that $\widetilde x_k\sim x_k$ and, by Definition~\ref{def:tildeAk}, 
$x_k$ is connected to $z$ in $\mathcal V^u\cap\mathcal D_k$. 
Therefore, $z$ is connected to $\edges(x_k')\subset \mathcal C_k$ (recall \eqref{eq:Ck} and Definition~\ref{def:Ak}) 
in $\mathcal V^u\cap(\mathcal D_k\cup\cube_k)$, and 
the event $\mathcal A_{z,k}$ does not occur. 
In other words, $\rho^\inn$ satisfies \eqref{eq:Pkbad:2}. 

\bigskip

It remains to show that $\rho^\inn$ satisfies \eqref{eq:Pkbad:1}. 
Remember that the total number of visits of all the $\rho^\inn_{i,j}$ to $\intb\cube_k$ is at most $R^{d-1}$. 
In particular, the total number of trajectories $\rho^\inn_{i,j}$ in $\rho^\inn$ is at most $R^{d-1}$, namely 
\begin{equation}\label{eq:rhoinn:1}
\sum_{i=1}^{N_{\cube_k,u}} M_{i} \leq R^{d-1} .\
\end{equation}
Since each $\rho^\inn_{i,j}$ is a self-avoiding path in $\cube_k$, 
\begin{equation}\label{eq:rhoinn:2}
|\rho^\inn_{i,j}| \leq |\cube_k| \leq (2R+1)^d .\
\end{equation}
Finally, observe that for any $x'\in\ZZ$ and vector $\tau^\inn\in\mathcal T^\inn_{x'}$, 
\begin{equation}\label{eq:rhoinn:3}
\widetilde {\mathcal A}_{z,k}\cap\{x_k' = x'\}\cap\{\rho^\inn = \tau^\inn\} 
\in 
\sigma(\Sigmagood,\Sigma_k,\mathcal F^\outt_{\cube(x'),u}) .\
\end{equation}
We get 
\begin{multline*}
\mathds{1}_{\widetilde {\mathcal A}_{z,k}\cap\{x_k' = x'\}}\cdot
\mathbb P\left[X^\inn_{\cube_k} = \rho^\inn~|~\sigma(\Sigmagood,\Sigma_k,\mathcal F^\outt_{\cube(x'),u}) \right]\\
\stackrel{\eqref{eq:Tinnx':meas}, \eqref{eq:Pkbad:0}, \eqref{eq:rhoinn:3}}
=
\sum_{\tau^\inn\in\mathcal T^\inn_{x'}}
\mathds{1}_{\widetilde {\mathcal A}_{z,k}\cap\{x_k' = x'\}\cap\{\rho^\inn = \tau^\inn\}}\cdot
\mathbb P\left[X^\inn_{\cube_k} = \tau^\inn~|~\sigma(\Sigmagood,\Sigma_k,\mathcal F^\outt_{\cube(x'),u}) \right]\\
\stackrel{\eqref{eq:Pk:lowerbound}, \eqref{eq:Pkbad:0}, \eqref{eq:rhoinn:1}, \eqref{eq:rhoinn:2}}
\geq
\mathds{1}_{\widetilde {\mathcal A}_{z,k}\cap\{x_k' = x'\}}\cdot
(1/2d)^{R^{d-1}\cdot (2R+1)^{d}} .\
\end{multline*}
This proves that $\rho^\inn$ satisfies \eqref{eq:Pkbad:1} with $\gamma = (1/2d)^{R^{d-1}\cdot (2R+1)^{d}}$. 
The proof of Lemma~\ref{l:Pkbad} is complete. 
\end{proof}

\subsection{Proof of \eqref{eq:sts:inerlacement:2}}\label{sec:proof:sts:int:2}

In this section we complete the proof of Theorem~\ref{thm:sts:interlacement} by showing how to deduce 
\eqref{eq:sts:inerlacement:2} from \eqref{eq:HN:proba} and \eqref{eq:strongsupercriticality}. 
We begin with the following lemma. 

\begin{lemma}\label{l:2crossings}
For $R$ and $u_1$ as in \eqref{eq:Ru_1}, there exist constants $c=c(d)>0$ and $C=C(d)<\infty$ such that 
for all $u<u_1$ and $N\geq 1$, 
\begin{equation}\label{eq:2crossings}
\mathbb P\left[
\begin{array}{c}
\text{$\ballZ(0,2(2R+1)N)\cap\mathcal V^u$ contains two nearest-neighbor paths}\\
\text{from $\ballZ(0,(2R+1)N)$ to $\intb\ballZ(0,2(2R+1)N)$ which are}\\
\text{in different connected components of $\ballZ(0,2(2R+1)N)\cap\mathcal V^u$}
\end{array}
\right]
\leq C\cdot e^{-N^c} .\
\end{equation}
\end{lemma}
\begin{proof}[Proof of Lemma~\ref{l:2crossings}]
Recall Definition~\ref{def:HN} of $\mathcal H_N$ and Definition~\ref{def:Ak} of $\mathcal A_{z,k}$. 
Note that when $\mathcal H_N$ occurs, the event in \eqref{eq:2crossings} implies that 
$\mathcal A_{z,k_N}$ occurs for some $z\in \ballZ(0,(2R+1)N)$. 
Therefore, we can bound the probability in \eqref{eq:2crossings} from above by 
\[
\mathbb P[\mathcal H_N^c] + \sum_{z\in\ballZ(0,(2R+1)N)}\mathbb P[\mathcal A_{z,k_N}] .\
\]
The result now follows from \eqref{eq:HN:proba} and \eqref{eq:strongsupercriticality}. 
\end{proof}
As an immediate corollary to Lemma~\ref{l:2crossings}, we obtain that for 
$u_1$ as in \eqref{eq:Ru_1} there exist constants $c = c(d)>0$ and $C = C(d)<\infty$ such that 
for all $u \leq u_1$ and $n\geq 1$, 
\begin{equation}\label{eq:2crossings:alln}
\mathbb P\left[
\begin{array}{c}
\text{$\ballZ(0,3n)\cap\mathcal V^u$ contains two paths from $\ballZ(0,n)$ to $\intb\ballZ(0,3n)$}\\
\text{which are in different connected components of $\ballZ(0,3n)\cap\mathcal V^u$}
\end{array}
\right]
\leq C\cdot e^{-n^c} .\
\end{equation}
We are now ready to prove \eqref{eq:sts:inerlacement:2}. 
Take $u_1$ as in \eqref{eq:Ru_1} and $u \leq u_1$. 
It suffices to consider $n\geq 100$. Let $k = \lfloor n/100\rfloor$. 
Note that if $\ballZ(0,n)\cap\mathcal V^u$ contains at least $2$ different connected components $\mathcal C_1$ and $\mathcal C_2$ 
with diameter $\geq n/10$, then 
there exist two vertices $x_1,x_2\in\ballZ(0,n)$ (possibly equal) such that 
$\mathcal C_i\cap\ballZ(x_i,k)\neq\emptyset$ and $\mathcal C_i\setminus\ballZ(x_i,7k)\neq\emptyset$, for $i\in\{1,2\}$. 

For $x\in\ballZ(0,n)$, let $\mathcal A_x$ be the event that 
\begin{itemize}
\item[(a)]
$\ballZ(x,k)$ is connected to $\intb\ballZ(x,7k)$ in $\mathcal V^u$, and 
\item[(b)]
every two nearest-neighbor paths from $\ballZ(x,2k)$ to $\intb\ballZ(x,6k)$ in $\mathcal V^u$ 
are in the same connected component of $\mathcal V^u\cap\ballZ(x,6k)$. 
\end{itemize}
Let $\mathcal A = \cap_{x\in\ballZ(0,n)}\mathcal A_x$. 
By \eqref{eq:sts:interlacement:1} and \eqref{eq:2crossings:alln}, we have 
\[
\mathbb P\left[\mathcal A\right]\geq 1 - C\cdot e^{-n^c} .\
\]
However, if the event $\mathcal A$ occurs, then $\mathcal C_1$ and $\mathcal C_2$, defined earlier, cannot exist. 
Indeed, take a nearest-neighbor path $\pi = (z_1,\ldots,z_t)$ in $\ballZ(0,n)$ from $x_1$ to $x_2$. 
For each $1\leq i\leq t-1$, the occurrence of the events $\mathcal A_{z_i}$ and $\mathcal A_{z_{i+1}}$ implies that 
(a) there exist nearest-neighbor paths $\pi_1$ and $\pi_2$ in $\mathcal V^u$, 
$\pi_1$ from $\ballZ(z_i,k)$ to $\intb\ballZ(z_i,7k)$, and 
$\pi_2$ from $\ballZ(z_{i+1},k)$ to $\intb\ballZ(z_{i+1},7k)$, and 
(b) any two such paths are in the same connected component of $\mathcal V^u\cap\ballZ(0,2n)$. 
This implies that $\mathcal C_1$ and $\mathcal C_2$ must be connected in $\mathcal V^u\cap\ballZ(0,2n)$. 
As a result, we have 
\[
\mathbb P\left[
\begin{array}{c}
\text{any two connected subsets of $\mathcal V^u\cap\ballZ(0,n)$ with}\\
\text{diameter $\geq n/10$ are connected in $\mathcal V^u\cap\ballZ(0,2n)$}
\end{array}
\right]
\geq 
\mathbb P\left[\mathcal A
\right] \geq 1-C\cdot e^{-n^c} .\
\]
This implies \eqref{eq:sts:inerlacement:2}. 
The proof of Theorem~\ref{thm:sts:interlacement} is completed.

\section{Extensions to other models}\label{sec:bonus}

\subsection{Random walk on $\Z^d$}

Consider a simple random walk on $\Z^d$, $d\geq 3$, started at $x\in\Z^d$. 
The random walk is transient, and the probability that $y\in\Z^d\setminus\{x\}$ is 
ever visited by the random walk is comparable to $|x-y|^{2-d}$. 

Let $\mathcal V$ be the set of vertices which are never visited by the random walk. 
The approach that we develop in this paper also applies to the study of 
the local connectivity properties of $\mathcal V$. 
Similarly to the proof of Theorem~\ref{thm:sts:interlacement}, one can show that 
the set $\mathcal V$, viewed as a random subgraph of $\Z^d$, 
contains a unique infinite connected component, which is also locally unique. 
Namely, the statements \eqref{eq:sts:interlacement:1} and \eqref{eq:sts:inerlacement:2} hold with 
$\mathcal V^u$ replaced by $\mathcal V$, and the law $\mathbb P$ of random interlacements replaced
by the law of a simple random walk started from $x\in\Z^d$. 

\subsection{Random walk on $(\Z/N \Z)^d$}

Consider a simple discrete time random walk on a $d$-dimensional torus $(\Z/N\Z)^d$, with $d\geq 3$. 
The vacant set at time $t$ is the set of vertices which have not been visited by the random walk up to time $t$. 
We view the vacant set as a (random) graph by drawing an edge between any two vertices of the vacant set at $L_1$-distance $1$ 
from each other. 
The study of percolative properties of the vacant set was initiated in \cite{BenjaminiSznitman} 
and recently significantly boosted in \cite{teixeira_windisch}. 
It was proved in \cite[Theorems 1.2 and 1.3]{teixeira_windisch} that the vacant set at time $\lfloor uN^d\rfloor$ 
exhibits different connectivity properties for small and large $u$: 
\begin{itemize}\itemsep1pt
\item[(i)]
if $u$ is large, there exists $\lambda = \lambda(u)<\infty$, such that the largest
 connected component of the vacant set at time $\lfloor uN^d\rfloor$ is 
smaller than $(\log N)^{\lambda}$ asymptotically almost surely, and 
\item[(ii)]
if $u>0$ is small, there exists $\delta = \delta(u)>0$, such that  
the largest connected component of the vacant set
 at time $\lfloor u N^d\rfloor$ is larger than $\delta N^d$ asymptotically almost surely, 
\end{itemize}
where ``asymptotically almost surely'' means ``with probability going to $1$ as $N\to\infty$''. 
Moreover, it is proved in \cite[Theorem~1.4]{teixeira_windisch} that when $d\geq 5$ and $u$ is small enough, with high probability, 
the vacant set on the torus at time $\lfloor uN^d\rfloor$ has the following properties: 
\begin{itemize}\itemsep1pt
\item[(a)]
the largest connected component has an asymptotic density, and 
\item[(b)]
the size of the second largest connected component is at most $(\log N)^\kappa$, for some $\kappa>0$. 
\end{itemize}
The proof of \cite[Theorem~1.4]{teixeira_windisch} relies on 
a strong coupling between random interlacements and the random walk trace (see \cite[Theorem~1.1]{teixeira_windisch}) 
and the existence of
\emph{strongly supercritical} values of $u$ for $d \geq 5$ 
(see \cite[Definition 2.4 and Remark 2.5]{teixeira_windisch}).
We believe that the ideas used in the proof of Theorem~\ref{thm:sts:interlacement}
can be applied in order to yield an extension of \cite[Theorem~1.4]{teixeira_windisch}
 for all $d\geq 3$ and small enough $u$, despite the fact that Theorem~\ref{thm:sts:interlacement}
does not imply the existence of strongly supercritical values of $u$.

\section{Appendix: Decoupling inequalities for interlacement local times}

In this appendix we prove Lemma \ref{l:probability:cascading}. 
The proof is essentially the same as the proof of \cite[Corollary~3.5]{Sznitman:Decoupling}. 
We sketch the main ideas here and refer the reader to corresponding formulas in \cite{Sznitman:Decoupling} for details.

\subsection{Notation from \cite[Section 1]{Sznitman:Decoupling}} 

For $K \subset \subset \Z^d$, we denote by $s_K: W^*_K \rightarrow W_K$ the map which 
associates with each element $w^* \in W^*_K$ the unique element $w^0 = s_K(w^*)\in W_K$ 
such that (a) $\pi^*(w^0) = w^*$ and (b)
$w^0(0)\in K$, $w^0(t)\notin K$ for all $t < 0$. 
For $w \in W$, we denote by $w_+$ the element in $W_+$ (see \eqref{def_eq_W_plus}) such that 
$w_+(n) = w(n)$, for $n\geq 0$.

For a finite measure $\rho$ on $\Z^d$, we denote by $P_{\rho}$ the measure
$\sum_{x \in \Z^d} \rho(x)P_x$ on $(W_+,\mathcal{W}_+)$.

Let $\omega= \sum_{i \geq 1} \delta_{(w^*_i,u_i)}$ be the interlacement point process on $W^*\times\R_+$ 
defined on the canonical probability space $(\Omega, \salgebra, \mathbb{P})$. 
For $K \subset \subset \Z^d$, and $0\leq u' < u$, 
we define on $(\Omega, \salgebra, \mathbb{P})$ the Poisson 
point processes on the space $W_+$ denoted by $\mu_{K,u}$ and $\mu_{K,u',u}$ in the following way:
\begin{equation}\label{def_eq_mu_K_u}
\begin{split}\itemsep1pt
\mu_{K,u',u} &=
 \sum_{i \geq 1}  \mathds{1}_{\{ w^*_i \in W^*_K, \; u' \leq u_i < u\}}  \delta_{s_K(w^*_i)_+}  \\
\mu_{K,u}&=\mu_{K,0,u} 
\end{split}
\end{equation}
With these definitions, we have (analogously to \cite[(1.27), (1.28)]{Sznitman:Decoupling}):
for $K\subset\subset\Z^d$ and $0\leq u'<u$, 
\begin{itemize}\itemsep1pt
\item[(i)]
$\mu_{K,u',u}$ and $\mu_{K,u'}$ are independent with respective intensity measures
$(u-u') P_{e_K}$ and $u' P_{e_K}$, 
\item[(ii)]
$\mu_{K,u} = \mu_{K,u'} + \mu_{K,u',u}$.
\end{itemize}

Let $I$ denote a finite or countable set. If $\mu =\sum_{i \in I} \delta_{w_i}$ is a point measure on $W_+$,
we define (by slightly abusing the notation of \eqref{def_eq_local_time_at_level_u})
the local time of $\mu$ at $x \in \Z^d$ to be
\begin{equation}\label{local_time_of_point_measure_of_trajectories}
\local_x(\mu)=  
\sum_{i \in I} \; \;  \sum_{n \in \N} \mathds{1}_{\{ w_i(n)=x \} } .\
\end{equation}

Using \eqref{def_eq_local_time_at_level_u}, \eqref{def_eq_mu_K_u} and \eqref{local_time_of_point_measure_of_trajectories}, 
we obtain that for any $\omega\in\Omega$, $K\subseteq K'\subset\subset \Z^d$, and $u\geq 0$, 
\begin{equation}\label{local_time_mu_K_u}
\local^u_x(\omega) = \local_x(\mu_{K',u}), \quad \text{for $x\in K$.} 
\end{equation}

\subsection{Decoupling inequalities for the interlacement local times}

In this section we extend the results of \cite[Section 2, 3]{Sznitman:Decoupling} about 
certain decoupling inequalities for increasing events in $\{0,1\}^{G\times\Z}$ to increasing events in $\loczd$. 
The graphs $G$ considered in \cite{Sznitman:Decoupling} are infinite, connected, bounded degree weighted graphs, 
satisfying certain regularity conditions, and in particular, include the case of $\Z^{d-1}$, with $d\geq 3$. 

\medskip

Since our current aim is to prove Lemma \ref{l:probability:cascading} on $\Z^d$, 
the notation of \cite{Sznitman:Decoupling} become slightly simpler. 
When $G=\Z^{d-1}$, 
the volume growth exponent of $G$ is $\alpha=d-1$, 
the diffusivity exponent of the random walk on $G$ is $\beta=2$, thus 
$\nu= \alpha -\frac{\beta}{2}=d-2$ is the ususal exponent of the Green function on $\Z^d$, c.f. \eqref{Green_decay}
and \cite[(0.2)]{Sznitman:Decoupling}.
Moreover, in \cite[(0.3)]{Sznitman:Decoupling}, a special metric $d(\cdot,\cdot)$ 
on $G \times \Z$ is introduced, but in our special case $\Z^d = G \times \Z$, 
the results of \cite{Sznitman:Decoupling} remain valid if we
replace the distance $d(x,x')$ by the usual sup-norm distance $|x-x'|$, 
 c.f. the first paragraph of \cite[Section 2]{Sznitman:Decoupling}.

\begin{remark}
The definition
\eqref{def_eq_local_time_at_level_u} carries over to the more general setting which involves local times of the 
interlacement point process on $G \times \Z$ (where $G$ satisfies the conditions described in \cite[Section 1]{Sznitman:Decoupling}),
and in fact all the results and proofs of \cite[Section 2, 3]{Sznitman:Decoupling} have their analogous, 
more general counterparts which involve $\local^u$ rather than $\mathcal{I}^u$.
To simplify the notation, we only consider the special case of $G = \Z^{d-1}$ here. 
\end{remark}

Now we recall some notation from  \cite[Section 2]{Sznitman:Decoupling}, which we adapt to our setting.

Our definition of the length scales $L_n=l_0^n L_0$ in \eqref{def:scalesLn} 
is the same as \cite[(2.1)]{Sznitman:Decoupling}.

For $n \ge 0$, we denote the dyadic tree of depth $n$  by $T_n = \bigcup_{0 \le k \le n} \{1,2\}^k$
  and the set of vertices of the tree
 at depth $k$ by $T_{(k)} = \{1,2\}^k$. We call $\emptyset \in T_{(0)}$ the root of $T_n$ and
$1,2 \in T_{(1)}$ the children of the root.
Given a mapping $\mathcal{T}: T_n \to \Z^d$, we define
\[
x_{m, \mathcal T} 
= \mathcal T(m), \;\; \widetilde C_{m, \mathcal T} 
= \ballZ(x_{m, \mathcal T}, \, 10 L_{n-k}), \;\mbox{for} \; m \in T_{(k)}, \; 0 \le k \le n\,.
\]
For any $0 \le k < n$, $m \in T_{(k)}$, we say that  $m_1,m_2$ are the two descendants of
 $m$ in $T_{(k+1)}$ if they are obtained by respectively concatenating $1$ and $2$ to $m$.
We say that $\mathcal{T}$ is an admissible embedding if for any $0 \le k < n$ and $m \in T_{(k)}$, 
\[
\widetilde C_{m_1,\mathcal T} \cup \widetilde C_{m_2,\mathcal T} 
\subseteq 
\widetilde C_{m,\mathcal T}\,, \qquad
|x_{m_1,\mathcal T}- x_{m_2,\mathcal T}| \ge \frac{1}{100} \;L_{n-k}\,. 
\]
For any $x \in \Z^d$ and $n\in \N$, we 
denote by $\Lambda_{x,n}$ the set of admissible embeddings of $T_n$ in $\Z^d$ with $\mathcal{T}(\emptyset)=x$, 
and let $\Lambda_n = \cup_{x \in \Z^d} \Lambda_{x,n}$.

\medskip
Recall the definition of the space $(\loczd,\loczdsig)$ and the coordinate maps $\Psi_x$, $x \in \Z^d$
 from Section \ref{subsection_local_times}.
Given $n \ge 0$ and $\mathcal{T} \in \Lambda_n$, we say that 
a collection $(B_m~:~m \in T_{(n)})$ of $\loczdsig$-measurable subsets of $\loczd$
     is  $\mathcal{T}$-adapted if 
\begin{equation}\label{eq:2.5}
\mbox{$B_m$ is $\sigma(\Psi_x, x \in \widetilde C_{m,\mathcal T})$-measurable for each $m \in T_{(n)}$}\,.
\end{equation}
Recall that given $u \ge 0$, the collection of $\salgebra$-measurable events 
$(B^u_m~:~m \in T_{(n)})$ is defined by \eqref{def:Au}.

\medskip

For $n \ge 0$ and $\mathcal T \in \Lambda_{n+1}$, we
 denote by $\mathcal T_1\in \Lambda_n$ the embedding of $T_n$ 
 corresponding to the restriction of $\mathcal T$ to the descendants 
of $1 \in T_{(1)}$ in $T_{n+1}$. We define $\mathcal T_2$ similarly using $2 \in T_{(1)}$.
Given a $\mathcal T$-adapted collection $(B_m~:~m \in T_{(n+1)})$,
we then define the $\mathcal T_1$-adapted collection $(B_{m,1}~:~m \in T_{(n)})$ 
and the $\mathcal T_2$-adapted collection  $(B_{m,2}~:~m \in T_{(n)})$
 in a natural way.

\medskip

We can now restate and adapt \cite[Theorem 2.1]{Sznitman:Decoupling} to fit our setting related 
to  $\local^u$ on $\Z^d$.

\begin{theorem}\label{theo2.1} 

\medskip
There exist $c=c(d)>0$ and $c_1=c_1(d)>0$  such 
that for all $l_0 \ge c$, $n \ge 0$,  $\mathcal{T} \in \Lambda_{n+1}$, any
$\mathcal{T}$-adapted collection $(B_m~:~m \in T_{(n+1)})$ of increasing events
on $(\loczd, \loczdsig)$, and any $0 < u' < u$ satisfying
\[
u \ge \big(1 + c_1  \;(n+1)^{-\frac{3}{2}} \,l_0^{-\frac{d-2}{4}}\big) \,u'\,,
\]
we have
\begin{equation}\label{our_decoupling}
\mathbb P \Big[ \bigcap\limits_{m \in T_{(n+1)}} B^{u'}_m\Big]  
\leq 
\mathbb P \Big[\bigcap\limits_{\overline{m}_1 \in T_{(n)}} B^u_{\overline{m}_1,1}\Big] 
\mathbb P \Big[\bigcap\limits_{\overline{m}_2 \in T_{(n)}} B^u_{\overline{m}_2,2}\Big] 
+ 
2 \exp\left(-2u' \, \frac{2}{(n+1)^3} \,L^{d-2}_n \, l_0^{\frac{d-2}{2}}\right) .\
\end{equation}
\end{theorem}

\begin{proof}
 The proof is analogous to that of \cite[Theorem 2.1]{Sznitman:Decoupling}. 
We only need to mechanically replace events defined in terms of 
$\mathcal{I}^u$ (see \eqref{def_eq_interlacement_at_level_u}) 
by events defined in terms of $\local^u$ (see \eqref{def_eq_local_time_at_level_u}).

When we adapt \cite[Theorem 2.1]{Sznitman:Decoupling} to suit our purposes, we make the 
following choices:  $\Z^d = G \times \Z$,  $G=\Z^{d-1}$, $\alpha=d-1$,  $\beta=2$,  
$\nu=d-2$, we use the sup-norm distance $|x-x'|$ on $\Z^d$ 
(c.f. the first paragraph of \cite[Section 2]{Sznitman:Decoupling}), moreover we choose
$K=2$ and $\nu' =\frac{d-2}{2}$ (where the latter parameters appear in the statement of
\cite[Theorem 2.1]{Sznitman:Decoupling}).

From \cite[(2.11)]{Sznitman:Decoupling} to \cite[(2.59)]{Sznitman:Decoupling}, we do not need to modify 
the proof at all, but we recall some further notation  before we state
 the key domination result \eqref{eq:2.59}.

Given $n \ge 0$ and 
$\mathcal{T} \in \Lambda_{n+1}$,  
we define, as in \cite[(2.11) and (2.13)]{Sznitman:Decoupling}, 
\[
\widehat C_i = \bigcup_{m \in T_{(n)}}
\widetilde C_{m,\mathcal T_i}, \;\; \mbox{for} \;\; i \in \{1,2\}, 
\;\; \mbox{and} \;\; 
V = \widehat C_1 \cup \widehat C_2\,,
\]
and 
\[
U_i = \ballZ \Big(x_{i,\mathcal T}, \; \frac{L_{n+1} }{1000}\Big), \;\; \mbox{for} \;\; i \in \{1,2\}, 
\;\; \mbox{and} \;\; 
U = U_1 \cup U_2 \,.
\]
Finally, we take a set $\setW \subset \Z^d$ such that $V \subseteq \setW \subseteq U$. 
Recall the notation \eqref{def:HU} and \eqref{def:TU}. 
For a trajectory in $W_+$ (see \eqref{def_eq_W_plus}), 
we define the sequence of successive returns to $\setW$ and departures from $U$:
\begin{equation}\label{eq:2.16}
\begin{split}
R_1 & = H_\setW, \, D_1 = T_U \circ \theta_{R_1} + R_1, \;\; \mbox{and by induction}
\\ 
R_{k+1} & = R_1 \circ \theta_{D_k} + D_k, \, D_{k+1} = D_1 \circ \theta_{D_k} + D_k , \;\mbox{for} \; k \ge 1\,,
\end{split}
\end{equation} 
where it is understood that if $R_k =\infty$ for some $k\geq 1$, then $D_k=R_{k+1}=\infty$.
Let $0\leq u'<u$. 
Recalling \eqref{def_eq_mu_K_u},  we introduce, similarly to \cite[(2.17)]{Sznitman:Decoupling}, the Poisson point processes on $W_+$,
\[
\begin{split}
\zeta_l' & = \mathds{1}_{\{ R_l < \infty = R_{l + 1}\}} \,\mu_{\setW,u'}, \;\; \mbox{for $l \ge 1$}\,,
\\ 
\zeta_l^* & = \mathds{1}_{\{ R_l < \infty = R_{l + 1}\}} \,\mu_{\setW,u',u}, \;\; \mbox{for $l \ge 1$}\,.
\end{split}
\]
Both $ \zeta_l'$ and $\zeta_l^* $ are supported on the subspace of $W_+$ which consists 
of trajectories that perform exactly $l$ returns to $\setW$ in the sense of \eqref{eq:2.16}.
 By the properties of $\mu_{\setW,u'}$ and $\mu_{\setW,u',u}$, 
\begin{equation}\label{eq:2.18}
\text{$\zeta_l'$, $l \ge 1$, and $\zeta_1^*$ are independent Poisson point processes on $W_+$}.
\end{equation} 
Recalling \eqref{local_time_of_point_measure_of_trajectories}, we define the local times 
\[
\begin{split}
\local_{l,x}' 
=
\local_x(\zeta_l'), \qquad
&\local_{l}' 
=
\left( \local_{l,x}'~:~x \in V\right), \quad \text{ for } l \geq 1,\\
\local^*_{1,x} 
=
\local_x( \zeta^*_1), \qquad
&\local_{1}^* 
= \left( \local^*_{1,x}~:~x \in V\right) .
\end{split}
\]
These definitions are counterparts of \cite[(2.60) and (2.61)]{Sznitman:Decoupling}. 
It follows from \eqref{eq:2.18} and \eqref{local_time_mu_K_u} that 
\begin{equation}\label{eq:2.62}
\begin{array}{c}
\text{the random variables $\local_l'$, $l\geq 1$, and $\local_1^*$ are independent,}\\
\text{$\local^{u'}_x = \sum_{l\geq 1} \local_{l,x}'$, and $\local^u_x \geq \local_{1,x}^*+\local_{1,x}'$, for all $x\in V$.}
\end{array}
\end{equation}
This is analogous to \cite[(2.62)]{Sznitman:Decoupling}. 
Moreover, similarly to \cite[(2.64)]{Sznitman:Decoupling}, we have that
\begin{equation}\label{eq:2.63}
\text{$(\local_{1,x}^* + \local_{1,x}'~:~x\in\widehat C_1)$ and 
$(\local_{1,x}^* + \local_{1,x}'~:~x\in\widehat C_2)$ are independent.}
\end{equation}

\medskip

The main ingredients in the proof of \cite[Theorem~2.1]{Sznitman:Decoupling} are \cite[Lemma~2.4 and (2.59)]{Sznitman:Decoupling}. 
We will only use a weaker result that immediately follows from \cite[Lemma~2.4 and (2.59)]{Sznitman:Decoupling}:
for a specific choice of $\setW$ 
(see \cite[(2.15) and (2.58)]{Sznitman:Decoupling}), 
there exists a coupling $(\overline\local',\overline\local^*)$ on 
$(\overline\Omega, \overline{\salgebra} ,\overline{\mathbb P})$ of 
$\sum_{l\geq 2}\local_l'$ and $\local_1^*$ such that 
\begin{equation}\label{eq:2.59}
\begin{array}{l}
\mbox{if $l_0 \ge c(d)$ and 
$u \ge \big(1 + c_1  \;(n+1)^{-\frac{3}{2}} \,l_0^{-\frac{d-2}{4}}\big) \,u'$, then}
\\[2ex]
\overline{\mathbb P}\left[ \overline\local' \leq   \overline{\local}^* \right] 
\geq 1- 2 \exp\left( -  u' \; \frac{4}{(n+1)^3} \;L^{d-2}_n \, l^{\frac{d-2}{2}}_0 \right).
\end{array}
\end{equation}
Informally, \eqref{eq:2.59} states that with high probability, the local times in $V$ of the collection of interlacement
 trajectories 
  which have labels less than $u'$ 
and reenter $\setW$ after leaving $U$
are dominated by the local times in $V$ of the collection of interlacement trajectories with labels between $u'$ and $u$ 
that never reenter $\setW$ after leaving $U$.

\bigskip

We now prove  \eqref{our_decoupling} by mimicking \cite[(2.68)]{Sznitman:Decoupling}.
We recall the notation from (\ref{def:Au}). 
Let $l_0$, $u$ and $u'$ satisfy \eqref{eq:2.59}. 
Since the $B_m$, $m \in T_{(n+1)}$, are increasing and $\mathcal{T}$-adapted, cf. \eqref{eq:2.5}, we see that

\begin{multline*}
\mathbb P \Big[\bigcap\limits_{m \in T_{(n+1)}} B^{u'}_m\Big] 
\stackrel{\eqref{eq:2.62}}{=} 
\mathbb P \Big[\bigcap\limits_{m \in T_{(n+1)}} B_m \Big( \sum_{l=1}^{\infty} \local'_{l}  \Big)\Big] \\
\stackrel{\eqref{eq:2.62}, \eqref{eq:2.59}}{\leq}
\mathbb P \Big[\bigcap\limits_{m \in T_{(n+1)}} B_m( \local_{1}^* + \local'_{1} )\Big] 
+2 \exp\left( -  u' \; \frac{4}{(n+1)^3} \;L^{d-2}_n \, l_0^{\frac{d-2}{2}} \right)\\
\stackrel{\eqref{eq:2.62},\eqref{eq:2.63} }{\leq}
\mathbb P \Big[\bigcap\limits_{\overline{m}_1 \in T_{(n)}} B^u_{\overline{m}_1,1}\Big] \,
\mathbb P  \Big[\bigcap\limits_{\overline{m}_2 \in T_{(n)}} B^u_{\overline{m}_2,2}\Big] 
+ 2 \exp\left( -  u' \; \frac{4}{(n+1)^3} \;L^{d-2}_n \, l_0^{\frac{d-2}{2}} \right) .\
\end{multline*}
This is precisely \eqref{our_decoupling}. 
\end{proof}

\bigskip

Now we derive the decoupling inequalities of \cite[Theorem 2.6]{Sznitman:Decoupling} 
adapted to our setting which involves local times.
Given $c_1$ and $l_0 \ge c$ as in Theorem \ref{theo2.1}, for any $u_0 > 0$ we define 
(analogously to \cite[(2.70)]{Sznitman:Decoupling}) 
\begin{equation}\label{eq:2.70}
u^-_\infty  
= 
u_0 \cdot \prod_{k=0}^\infty 
\Big(1 + \frac{c_1}{(k+1)^{3/2}} \; l_0^{-\frac{d-2}{4}} \Big)^{-1} .\
\end{equation}
Note that $u^-_\infty>0$ and $u^-_\infty \to u_0$ as $l_0 \to \infty$.

\begin{theorem}[Decoupling Inequalities]
  \label{theo2.6} 
For any $L_0\geq 1$, $l_0 \ge c(d)$,  $u_0 > 0$, $n \ge 0$, $\mathcal{T} \in \Lambda_n$, and all $\mathcal{T}$-adapted collections
 $(B_m~:~m \in T_{(n)})$ of increasing events on $(\loczd, \loczdsig)$, one has
\[
\mathbb P \Big[\bigcap\limits_{m \in T_{(n)}}\, B_m^{u^-_\infty}\Big]  
\leq  
\prod\limits_{m \in T_{(n)}}\, \left(\mathbb{P}[B_m^{u_0}] + \varepsilon(u^-_{\infty}, l_0, L_0)\right)\,,
\]
where 
\begin{equation}\label{eq:2.73}
\varepsilon(u,l_0, L_0) 
= 
f( 2 u L_0^{d-2} \, l_0^{\frac{d-2}{2}} ),
\;\; \mbox{ with  $f(v)= \frac{2\cdot e^{-v}}{1-e^{-v}}$ }\,.
\end{equation}
\end{theorem}
\begin{proof}
The proof of Theorem \ref{theo2.6} is identical to that of \cite[Theorem 2.6]{Sznitman:Decoupling}. 
We only need to make the particular choices  $K=2$, $\nu = d-2$ and  $\nu ' = \frac{d-2}{2}$, and 
replace references to \cite[Theorem 2.1]{Sznitman:Decoupling} by references to 
Theorem \ref{theo2.1}.
We omit the details.
\end{proof}

\bigskip

Recall the definition of uniformly cascading events from Definition \ref{def:cascading}.
We now restate \cite[Theorem 3.4]{Sznitman:Decoupling} adapted to our setting, 
which involves local times.

\begin{lemma}\label{l:theorem3.4} 
Consider the collection $\mathcal G = (G_{x,L,R})_{x\in \Z^d, L\geq 1, R\geq 0 }$ of increasing 
events on $(\loczd, \loczdsig)$,  cascading uniformly (in $R$) with complexity at most $\lambda$. 
Then for any $l_0 \ge c(d)$, $L_0 \geq 1$, $n \ge 0$, $u_0 > 0$ and $R\geq 0$, we have
\begin{equation}\label{eq:3.22}
\sup\limits_{x \in \Z^d} \, \mathbb P \left[G_{x,L_n,R}^{u_\infty^-}\right]  
\leq 
\left(C(\lambda)^2\cdot l_0^{2 \lambda}\right)^{2^n-1}
\left(\sup\limits_{x \in \Z^d} \, \mathbb P\left[G_{x,L_0,R}^{u_0}\right]   
+ \varepsilon(u_\infty^-,l_0, L_0)  \right)^{2^n}\,,
\end{equation}
where the constant $C(\lambda)$ was defined in \eqref{eq:cascading:3}.
\end{lemma}
\begin{proof}
The proof of Lemma \ref{l:theorem3.4} is identical to that of \cite[Theorem 3.4]{Sznitman:Decoupling}. 
We only need to make the particular choices  $K=2$, $\nu = d-2$ and  $\nu ' = \frac{d-2}{2}$, and 
note that the inequality \eqref{eq:3.22} holds uniformly in $R$ because the bound of \eqref{eq:cascading:3} holds uniformly in $R$.
We omit the details.
\end{proof}

\subsection{Proof of Lemma \ref{l:probability:cascading}}

We are now ready to prove Lemma~\ref{l:probability:cascading}, using Lemma~\ref{l:theorem3.4}. 
This is similar to the proof of \cite[Corollary~3.5]{Sznitman:Decoupling}. 

Let $\mathcal G= (G_{x,L,R})_{x\in \Z^d, L\geq 1, R\geq 0 }$ be a family 
of increasing events on $(\loczd, \loczdsig)$ cascading uniformly in $R$ with complexity at most $\lambda>0$. 
Recall the notation from \eqref{eq:2.70} and \eqref{eq:2.73}. 
We will choose $l_0 \ge c(d)$, $L_0 \geq 1$, $u_0 > 0$, and $R\geq 0$ so that 
\begin{equation}\label{half}
C(\lambda)^2\cdot l_0^{2 \lambda} \cdot
\left(
\sup\limits_{x \in \Z^d} \, \mathbb P \left[G_{x,L_0,R}^{u_0}\right] 
+ \varepsilon(u^-_\infty, l_0, L_0)  
\right) 
\leq \frac12 .\
\end{equation}
Once we do so, \eqref{eq:probability:cascading:2} will immediately follow from Lemma \ref{l:theorem3.4} 
with $l_0$, $L_0$, and $R\geq 0$ as in \eqref{half} and $u = u^-_\infty$. 

\medskip

Let $u_0 = u_{L_0} = {L_0}^{2-d}$. By \eqref{eq:2.70} and \eqref{eq:2.73}, 
for all large enough $l_0\geq c(d)$, we have  
\begin{equation}\label{quarter_1}
\sup_{L_0 \geq 1 }  C(\lambda)^2\cdot l_0^{2 \lambda} \cdot \varepsilon( u^-_\infty  ,l_0, L_0) \leq \frac14.
\end{equation}
We fix $l_0$ satisfying \eqref{quarter_1}.
Now we use our assumption \eqref{eq:probability:cascading:1} to choose $L_0 \geq 1$ and $R \geq 0$ such that
\begin{equation}\label{quarter_2}
C(\lambda)^2\cdot l_0^{2 \lambda} \cdot
\sup\limits_{x \in \Z^d} \, \mathbb P\left[G_{x,L_0,R}^{u_0}\right] \leq \frac14. 
\end{equation}
The combination of \eqref{quarter_1} and \eqref{quarter_2} gives \eqref{half} and finishes the proof of
Lemma \ref{l:probability:cascading}.
\qed

\paragraph{Acknowledgements.}
We thank A.-S. Sznitman for useful comments on the draft. 
We also thank the referee for careful reading of the paper. 
The work on this paper was done while the authors were at ETH Z\"urich. 
During that time, the research of AD was supported by an ETH Fellowship, and  
the research of BR and AS was supported by the grant ERC-2009-AdG 245728-RWPERCRI.

\end{document}